\documentclass[10pt]{paper}%
\usepackage[T1]{fontenc}
  \usepackage{amsmath}
  \usepackage[all,ps]{xy}
\usepackage{amssymb}
\usepackage{fullpage}
\usepackage{enumerate}
\usepackage{amsthm}
\usepackage{multirow}
\usepackage{tikz}
\usepackage{color}
\usepackage{pifont}
\usepackage{graphicx}
  \usepackage[all]{xy}
\usepackage{amsfonts}
\usepackage{amssymb}
\usepackage{ytableau}%

\newcommand{\N}{\mathbb{N}}

 \newcommand{\sym}{\mathfrak{S}}
 \newcommand{\alt}{\mathfrak{A}}
 
 \newcommand{\Z}{\mathbb{Z}}

 \newcommand{\IBr}{\operatorname{IBr}}
  \newcommand{\Ind}{\operatorname{Ind}}
  \newcommand{\Res}{\operatorname{Res}}
  \newcommand{\BB}{{B}}

\newtheorem{Th}{Theorem}[]
\newtheorem{lemma}[Th]{Lemma}

\newtheorem{Prop}[Th]{Proposition}

\theoremstyle{remark}
\newtheorem{Rem}[Th]{Remark}{\rmfamily}
\theoremstyle{definition}
\newtheorem{Def}[Th]{Definition}{\rmfamily}
\newtheorem{exa}[Th]{Example}{\rmfamily}

\newtheorem{Hyp}[Th]{Hypothesis}{\rmfamily}
\newtheorem{Not}[Th]{Notation}{\rmfamily}

\def\Z{\mathbb{Z}}

\newcommand{\Irr}{\operatorname{Irr}}

\newcommand\blfootnote[1]{%
  \begingroup
  \renewcommand\thefootnote{}\footnote{#1}%
  \addtocounter{footnote}{-1}%
  \endgroup
}
\begin{document}

\title{On unitriangular basic sets for symmetric and alternating groups}
\author{Olivier Brunat, Jean-Baptiste Gramain and Nicolas Jacon}
\maketitle
\date{}
\blfootnote{\textup{2020} \textit{Mathematics Subject Classification}: \textup{20C20;20C30}} 
\begin{abstract}
    We study the modular representation theory of the symmetric and
alternating groups. One of the most natural ways to label the irreducible
representations of a given group or algebra in the modular case is to show
the unitriangularity of the decomposition matrices, that is, the existence
of a unitriangular basic set. We study several ways to obtain such sets in the general case of a symmetric algebra. 
We apply our  results to  the symmetric groups and to their Hecke algebras and thus obtain new ways to label  
  the simple modules 
 for these objects.
 Finally, 
we show that these sets do not always
exist in the case of the alternating groups by studying   two
explicit cases in
 characteristic $3$. 
\end{abstract}

\section{Introduction}

   One of the major problems in the  representation theory of finite groups
is to find a natural classification of the set of simple modules, in
particular in the modular case. This problem may be attacked using the
notion of unitriangular basic sets in the wider context of symmetric
algebras.  Let $A$ be an integral domain and let $\mathcal{H}$ be an
associative symmetric $A$-algebra, finitely generated and free over $A$
(see \cite[Def.  7.1.1]{GP}). A standard example of such structure is the
group algebra of a finite group. Another example is the Iwahori-Hecke
algebra of a finite Weyl group. Let $\theta: A \to L$ be a ring
homomorphism into a field $L$ such that $L$ is the field of fractions of
$\theta (A)$.
   We obtain an $L$-algebra $L\mathcal{H}:=L\otimes_A \mathcal{H}$, where
$L$ is regarded as an $A$-module via $\theta$. We assume that
$L\mathcal{H}$ is split. Our problem is then to describe the set
$\operatorname{Irr}(L\mathcal{H})$ of isomorphism classes of simple
$L\mathcal{H}$-modules. Note that when $\mathcal{H}$ is a group algebra of
a finite group and $L$ is a sufficiently large field of characteristic $p$
dividing the order of the group, this problem coincides with that of
finding the irreducible modular representations of the finite group.

    A useful tool in this setting is the decomposition matrix. Let $K$ be
the field of fractions of $A$ and write $K\mathcal{H}:=K\otimes_A
\mathcal{H}$. We assume that $A$ is integrally closed in $K$ and that $K
\mathcal{H}$ is split semisimple. Assume that we get a classification of
the simple $K \mathcal{H}$-modules
\begin{equation}
\label{eq:notirr}  
 \operatorname{Irr}(K\mathcal{H})=\{ V^{\lambda}\ |\ \lambda \in \Lambda\}
\end{equation}
for some labeling set $\Lambda$. The decomposition map is then the map (see
\cite[Theorem 7.4.3]{GP})
$$d_{\theta}: R_0 (K \mathcal{H}) \to R_0 (L  \mathcal{H})$$
between the associated Grothendieck groups of finite dimensional modules.
In particular, for all $\lambda \in \Lambda$ and
$M\in\operatorname{Irr}(L \mathcal{H})$, there are uniquely determined non-negative
integers $d_{\lambda,M}$ such that
\begin{equation}
\label{eq:decmat}
d_{\theta}([V^{\lambda}])=\sum_{M\in \operatorname{Irr} (L \mathcal{H})} d_{\lambda, M}[M].
\end{equation}
    We then get the corresponding \emph{decomposition matrix} $(d_{\lambda,
M})_{\lambda \in \Lambda, M\in \operatorname{Irr}(L \mathcal{H})}$, which controls the
representation theory of $L \mathcal{H}$. 
  
\begin{Def}
A {\it unitriangular basic set for} $( \mathcal{H},\theta)$ is the datum of a triplet
$(\BB,\leq,\Psi)$ where $\BB\subset \Lambda$, $\leq$ is a total order
defined on $\Lambda$ and $\Psi$ is a bijective map: $$\Psi: \BB \to
\operatorname{Irr} (L\mathcal{H})$$ satisfying:
\begin{enumerate}
\item for all $\lambda \in \BB$, we have $d_{\lambda,\Psi (\lambda)}=1$,
\item for all $M\in \operatorname{Irr} (L \mathcal{H})$ and $\lambda\in \Lambda$, we
have $d_{\lambda,M}=0$ unless $\lambda \leq \Psi^{-1} (M)$.
\end{enumerate}
\end{Def}
    The existence of such a set can be helpful in order to solve our main
problem. Indeed, we first remark that 
$$ \operatorname{Irr} (L \mathcal{H})=\{\Psi(\lambda)  \ |\ \lambda \in
\BB\},$$
and the unitriangularity implies that the rows (labeled by $\Lambda$) and
columns (labeled by $\operatorname{Irr} (L \mathcal{H})$) of the associated
decomposition matrix can be ordered to get a unitriangular shape as
follows:

    {\tikzstyle{every picture}+=[remember picture]
\renewcommand{\arraystretch}{1.2}
\renewcommand{\arraycolsep}{2pt}
$$\begin{array}{c|p{1pt}ccccccp{1pt}|}
\cline{2-2}
\cline{9-9}
&&1&0&\cdots&0&\cdots&0&\\
&&*&\ddots&\ddots&\vdots&&\vdots&\\
&&*&&&1&&\vdots&\\
&&*&&&&&1&\\
&&*&&&&&*&\\
&&\vdots&&&&&\vdots&\\
\cline{2-2}
\cline{9-9}
\end{array}.
$$}
\medskip

    This then gives a natural and canonical way to label the simple 
modules of $L \mathcal{H}$ by the set $\BB$. Thus, it is an important problem to
construct, if possible, unitriangular basic sets associated to a given
symmetric algebra. In this paper, we will mainly focus on the case where
$ \mathcal{H}$ is the group algebra of the alternating or symmetric group,
or the Hecke algebra of the symmetric group.  In the first two cases, the field $L$ will be 
 a field of positive characteristic and the decomposition matrix given above
 corresponds to the usual ``$p$-decomposition matrix''  for a finite group. 
\medskip

    In \cite{BG}, the first two authors have already shown the existence
of a weak version of unitriangular basic sets (simply called {\emph{basic sets}})
for alternating groups. It was expected that this result may be
strengthened by exhibiting unitriangular basic sets. In
Section~\ref{sec:alterne}, we will in fact show that this is not possible
in general. This result is surprising and unexpected, especially when we
compare with the case of the symmetric groups where it is well known that
unitriangular basic sets do exist \cite{JK}. As usual, the representations
of the alternating groups can be obtained from those of the symmetric
groups using Clifford theory. We will recall several aspects of this
theory in relation with these unitriangular basic sets, and in particular
deduce  a criterion for the existence of unitriangular basic sets for
alternating groups. 
    Then, using this, we deduce two counterexamples for the existence of
unitriangular basic sets for $n=18$ and $n=19$ and $p=3$ by two different
methods Furthermore, the
definition of unitriangular basic set may be refined using the well known
partition of the simple modules into $p$-blocks. In the cases of the
symmetric and alternating groups, we can attach to each $p$-block a
non-negative integer called its {\emph{$p$-weight}}. The structure of two
$p$-blocks with the same $p$-weight is very similar. For example, Chuang
and Rouquier \cite{CR} proved that two such $p$-blocks of symmetric groups
are derived equivalent, and, in~\cite{BG}, a weaker result for $p$-blocks
of the alternating groups with the same weight is shown. We can then
expect that two $p$-blocks of the alternating groups with the same
$p$-weight should have the same unitriangularity property together. But we
will show that this is not the case. 

    Despite our counterexamples for the alternating groups, one can ask in
which cases and for which types of blocks unitriangular basic sets exist.
To attack this problem, our strategy consists in focusing on the symmetric
group, where we will be interested in the existence of unitriangular basic
sets of a special kind. 

    There is also another (connected) motivation to search for these sets.
It is well known that the ordinary irreducible representations of the
symmetric group $\mathfrak{S}_n$ are naturally parametrized by the set
$\Pi^1 (n)$ of partitions of $n$, that is, the set of non-increasing
sequences of non-negative integers of total sum $n$. In addition, these
representations, or equivalently the simple $K \mathfrak{S}_n$-modules
(where $K$ is a field of characteristic $0$), can be constructed
explicitly thanks to the Specht module theory (for example) and we get
nice formulae for their dimensions. For $\lambda\in \Pi^1 (n)$, we write
$S^{\lambda}$ for the corresponding simple $K\mathfrak S_n$-module. We
have
\begin{equation}
\label{eq:sym}
\operatorname{Irr}(K \mathfrak{S}_n)=\{ S^{\lambda} \ |\ \lambda \in 
      \Pi^1 (n)\}.
\end{equation}
    When $L$ is a field of characteristic $p>0$, there is also a canonical
way to label the simple $L\mathfrak{S}_n$-modules by a certain subset $\operatorname{Reg}_p (n)$ of
partitions, called the {\emph{$p$-regular
partitions}}. These are the partitions in $\Pi^1 (n)$ which do not have $p$ or more
parts of the same positive size. A complete set of simple
$L\mathfrak{S}_n$-modules is then obtained as certain non zero quotients
of the Specht modules:
$$\operatorname{Irr} (L\mathfrak{S}_n)=\{ D^{\lambda} \ |\ \lambda \in 
  \operatorname{Reg}_p (n) \}.$$
  
    In characteristic zero, the one-dimensional representations of the
symmetric group are naturally labeled by the partition $(n)$ (for the
trivial representation) and its conjugate partition $(\underbrace{1,
\ldots, 1}_{n \text{ times}})$ (for the sign representation). In positive
characteristic $p>2$, the trivial representation is still labeled by the
partition $(n)$ but the partition $(\underbrace{1,\ldots,1}_{n \text{
times}})$ is not $p$-regular in general so it cannot label the sign
representation. In fact, the problem of finding which $p$-regular
partition labels the sign representation is a particular case of a problem
raised by Mullineux in \cite{Mu}. The problem is the following. If
$\lambda$ is a partition, then it is well known that tensoring the simple
module $S^{\lambda}$ by the sign representation leads to a simple module
isomorphic to $S^{\lambda'}$ where $\lambda'$ is the conjugate of
$\lambda$. This construction still makes sense in positive characteristic
for the simple modules $D^{\lambda}$ indexed by the $p$-regular
partitions. However, the analogue of the conjugate partition is here more
difficult to obtain. In fact, we now know that it can be described by
several non-trivial recursive algorithms \cite{K,FK,X} (the first one
having been conjectured by Mullineux \cite{Mu}).

    It is a natural question to ask if one can obtain a classification of
the simple modules for which the tensor product by the sign representation is
easier to describe. One of our results will be  to give an answer to this
problem using  the theory of unitriangular basic sets.\medskip

    Let us now explain in detail the organization of the paper. First, in
Section~\ref{sec:indice2}, we consider the general situation where
$G$ is a finite group, $H$ is an index two subgroup of $G$ and $p$ is an
odd prime number dividing the order of $G$. We denote by
$\varepsilon:G\rightarrow \{-1,1\}$ the linear character of $G$ obtained
by inflating the faithful linear character of $G/H$ to $G$, and by
$\sigma:H\rightarrow H,\,g\mapsto xgx^{-1}$ an automorphism of $H$,
where $x\in G$ is a fixed element with $x\notin H$. In this context, we
will give relations between unitriangular $p$-basic sets of $G$ and of $H$.
More precisely, we will show that if $\mathfrak B$ is a union of
$p$-blocks that covers a $\sigma$-stable union $\mathfrak b$ of $p$-blocks
of $H$ such that $\mathfrak b$ has a unitriangular $p$-basic set $b$, then
there is a unitriangular $p$-basic set $(B,\leq,\Theta)$ of $\mathfrak B$
satisfying the following two conditions:
\begin{enumerate}[(i)]
\item The set $B$ is $\varepsilon$-stable.
\item The map $\Theta:B\rightarrow \IBr_p(\mathfrak B)$ is
$\varepsilon$-equivariant, where $\IBr_p(\mathfrak B)$ is the set of irreducible
Brauer characters of $\mathfrak B$.
\end{enumerate}
Furthermore, a unitriangular $p$-basic set of $\mathfrak B$ that satisfies
these properties restricts to a unitriangular $p$-basic set of
$\mathfrak b$. See Theorem~\ref{thm:basicGH}, Theorem~\ref{thm:basicHG}
and Remark~\ref{rk:equivalence}.

    In Section~\ref{sec:alterne},  we  study in detail the case of the
symmetric and alternating groups. We begin by applying our previous
results to this case. Then, in~\S\ref{subsec:contreexemples}, we give two
counterexamples for the existence of a unitriangular basic set for the
alternating groups, by showing that the principal $3$-blocks of $\alt_{18}$
and $\alt_{19}$ have no unitriangular $3$-basic set. 

    The aim of the last section is to give a general procedure to produce
unitriangular basic sets in the case of the symmetric and alternating groups
when this is possible. 
    In fact, we consider in Section~\ref{sec:symalgebras} this problem in
the more general setting of symmetric algebras. We propose a procedure to
obtain from a given unitriangular basic set, new such sets with nice
additional properties. This part is quite elementary but should be of
independent interest. In Section~\ref{sec:application}, we give some
applications of the above result to the symmetric group and to a well
known unitriangular basic set for this group, or more generally to its
Hecke algebra. Even these resulting new unitriangular $p$-basic sets are
not completely stable by tensoring by the sign of $\sym_n$ in general, but
they {\emph{almost}} are. We then give a way to modify these sets to
obtain $p$-basic sets satisfying conditions (i) and (ii) above. 
    Then, finally, we apply these  results to prove that any $p$-block of
$\alt_n$ with an odd $p$-weight has a unitriangular $p$-basic set. We also
study explicitly an example which shows  that ``unitriangularity'' is
not an invariant of the $p$-weight of the $p$-blocks of the alternating
groups.

\section{Groups with an index two subgroup}
\label{sec:indice2}

    In this section, we study the relations between the basic sets of a
group with those of an index two subgroup. Of course, we can have in
mind the symmetric and the alternating group as a fundamental example.
Theorem \ref{thm:basicGH} shows how one can  obtain a unitriangular basic
set for $H$ from a particular unitriangular basic set for $G$.  Theorem
\ref{eq:propcool} gives a necessary condition for the existence of a
unitriangular basic set for $H$. Both results will be crucial in the rest
of the paper. 

\subsection{Setting}

    Let $G$ be a finite group and $p$ be a prime number dividing $|G|$.
Assume that $A$ is a valuation ring such that $p$ belongs to its maximal
ideal. Suppose that its field of fractions $K$ is a splitting field for
$G$ containing $A$.  Consider the canonical map $\theta:A\to L$, where $L$
is the residue field of $A$. In this case, $(K,A,L)$ is a modular system
for $G$ and $p$ large enough in the
following. To any $LG$-module $M$, we can associate its Brauer character
$\varphi_M$. It is an $A$-valued class function which vanishes on the set of
$p$-singular elements of $G$. We write $\IBr_p(G)$ for the set of Brauer characters of irreducible $LG$-modules, and $\Irr(G)$ for the set
of (ordinary) irreducible characters of $G$. We assume that $\Lambda$  is
an indexing set for $\Irr(G)$, so that we have $\Irr(G) =
\{\chi_{\lambda}\mid \lambda\in\Lambda\}$. Furthermore, for any
$\varphi\in K\Irr(G)$, we define
\begin{equation}
\label{eq:defchapeau}
\widehat\varphi(g)=
\begin{cases}
\varphi(g)&\text{if $g$ is a $p$-regular element,}
\\ 
0&\text{otherwise.}
\end{cases}
\end{equation}
Recall that $\widehat{\chi}\in\N\IBr_p(G)$ for all $\chi\in\Irr(G)$, and that
 \begin{equation}
\label{eq:relbasique}
 \widehat\chi_{\lambda}=\sum_{M\in\Irr(L\mathcal{H})}d_{\lambda,M}\,\varphi_{M}.
\end{equation}
    The numbers $d_{\lambda,M}$ are the $p$-decomposition numbers, and the
associated matrix $(d_{\lambda,M})_{\lambda \in \Lambda, M \in
\Irr(L\mathcal{H})}$ is the decomposition matrix. The number
$d_{\lambda,M}$ is also denoted by $d_{\chi,\varphi}$ when $\chi$ is the
irreducible character of $G$ labeled by $\lambda$ and $\varphi$ is the
Brauer character of the simple $L\mathcal{H}$-module $M$.

    Now, for $\varphi\in\IBr_p(G)$, the projective indecomposable
character corresponding to $\varphi$ is the ordinary character
$\Phi_\varphi$ defined by
\begin{equation}
\label{eq:pims}
\Phi_\varphi=\sum_{\chi\in\Irr(G)}d_{\chi,\varphi}\chi.
\end{equation}
    Write $\operatorname{IPr}_p(G)$ for the set of projective
indecomposable characters of $G$. The set $\N\operatorname{IPr}_p(G)$ is
the set of projective characters of $G$. On the other hand, recall that
$\Z\operatorname{IPr}_p(G)$ is the set of generalized characters of $G$
vanishing on the set of $p$-singular elements. Recall that
$\operatorname{IPr}_p(G)$ is the dual basis of $\IBr_p(G)$ with respect to
the hermitian scalar product $\langle\,,\,\rangle_G$ on $K\Irr(G)$, that
is, the unique $K$-basis of the subspace of $K\Irr(G)$ vanishing on
$p$-singular elements such that $\langle \Phi_\varphi , \vartheta\rangle_G
= \delta_{\varphi,\vartheta}$ for all $\varphi,\,\vartheta\in\IBr_p(G)$.

\begin{lemma}
\label{lemma:automorphisme} We keep the notation as above. For any
$\sigma\in\operatorname{Aut}(G)$ and $\varphi\in\IBr_p(G)$, we have
$${}^{\sigma}\Phi_{\varphi}=\Phi_{{}^{\sigma}\varphi}.$$
\end{lemma}

\begin{proof} Assume that we have $\sigma\in\operatorname{Aut}(G)$. 
    By Equation~(\ref{eq:relbasique}) applied to ${}^{\sigma} \widehat{
\chi}$ and $\widehat{\chi}$, and using that ${}^{\sigma} \widehat{
\chi}=\widehat{ {  }^{\sigma}\chi}$, we obtain
$$\sum_{\varphi\in\IBr_p(G)}d_{\chi,\varphi}{}^{\sigma}\varphi= 
{}^{\sigma}\widehat{\chi}=
\widehat{{}^{\sigma}\chi}=\sum_{\varphi\in\IBr_p(G)}
d_{{}^{\sigma}\chi,\varphi}\,\varphi=
\sum_{\varphi\in\IBr_p(G)}d_{{}^{\sigma}\chi,{}^{\sigma} \,
\varphi}{}^{\sigma}\varphi.$$
Hence, the uniqueness of the coefficients in a basis gives
\begin{equation}
\label{eq:coeffauto}
d_{\chi,\varphi}=d_{{}^{\sigma}\chi,{}^{\sigma}\varphi} 
\end{equation}
for all $\chi\in\Irr(G)$ and $\varphi\in\IBr_p(G)$. Furthermore, for
$\varphi\in\IBr_p(G)$, Equation~(\ref{eq:coeffauto}) gives
$${}^{\sigma}\Phi_{\varphi} = \sum_{ \chi\in\Irr(G)} d_{
\chi,\varphi}{}^{\sigma}\chi
=\sum_{\chi\in\Irr(G)}d_{{}^{\sigma}\chi,{}^{\sigma}\varphi}\,{}^{\sigma}\chi
=\Phi_{{}^{\sigma}\varphi},$$
as required.
\end{proof}

    The {\emph{$p$-blocks}} of ordinary characters of $G$ are the equivalence classes of the
following equivalence relation on $\Irr(G)$: for $\chi,\,\psi\in\Irr(G)$, the characters $\chi$
and $\psi$ are in relation if and only if there are
$\varphi_0,\ldots,\varphi_r\in\IBr_p(G)$ such that $\chi \in\Phi_{
\varphi_0}$, $ \psi\in\Phi_{\varphi_r}$ and, for each $0\leq i\leq r-1$,
$$\langle\Phi_{\varphi_i},\Phi_{\varphi_{i+1}}\rangle_G\neq 0.$$ 
    The set of Brauer characters appearing in those PIMs form a $p$-block
of Brauer characters. In the following, a $p$-block of $G$ can refer to a
subset of $\Irr(G)$ or of $\IBr_p(G)$ according to the context. It  will
be denoted by $\Irr(\mathfrak B)$ or $\IBr_p(\mathfrak B)$, or sometimes
only by $\mathfrak B$ if there is no possible confusion.

\begin{lemma}
Let $\sigma\in\operatorname{Aut}(G)$ and $p$ be a prime number. Let
$\mathfrak B$ be a $p$-block of $G$. Then 
$${}^{\sigma}\Irr(\mathfrak B)=\{{}^{\sigma}\chi\mid \chi\in\Irr(\mathfrak
B)\}\quad\text{and}\quad
{}^{\sigma}\IBr_p(\mathfrak B)=\{{}^{\sigma}\varphi\mid \varphi\in\IBr_p(\mathfrak
B)\}$$
is a $p$-block of $G$, denoted by ${}^{\sigma}\mathfrak B$.
\end{lemma}

\begin{proof}
    First, we remark that, for any class functions $\alpha$ and $\beta$ of
$G$, we have
$$\langle \alpha,\beta \rangle_G=\langle
{}^{\sigma}\alpha,{}^{\sigma}\beta\rangle_G.$$
Now, $\chi$ and $\psi$ are in $\Irr(\mathfrak B)$ if and only if there are
$\varphi_0,\ldots,\varphi_r\in\IBr_p(G)$ such that $
\langle\Phi_{\varphi_i},\Phi_{\varphi_{i+1}}\rangle_G\neq 0$. Since
$$
\langle\Phi_{\varphi_i},\Phi_{\varphi_{i+1}}\rangle_G=
\langle {}^{\sigma}\Phi_{\varphi_i},{}^{\sigma}
\Phi_{\varphi_{i+1}}\rangle_G=
\langle \Phi_{{}^{\sigma}\varphi_i},
\Phi_{{}^{\sigma}\varphi_{i+1}}\rangle_G,$$
the result follows.
\end{proof}

\begin{Def}
\label{def:basicset}
With the above notation, a subset $B\subseteq\Irr(\mathfrak B)$ of a
$p$-block $\mathfrak B$ is a {\emph{$p$-basic set}} of $\mathfrak B$ if the set $\{\widehat{\chi} \mid \chi\in B\}$ is a $\Z$-basis of the $\Z$-module $\Z\IBr_p(\mathfrak
B)$. 
\end{Def}

\begin{Rem}
\label{rk:defbasic}
    Since the set $\{ \widehat{\chi} \mid \chi\in\Irr(\mathfrak B)\}$
generates over $\Z$ the module $\Z\IBr_p(\mathfrak B)$, a subset
$B\subseteq\Irr(\mathfrak B)$ is a $p$-basic set of the $p$-block
$\mathfrak B$ if and only if
\begin{itemize}
\item For any $\chi\in\Irr(\mathfrak B)$, $\widehat \chi$ is a $\Z$-linear
combination of the set $\widehat B=\{\widehat \psi\mid \psi \in B\}$.
\item The family $\widehat B$ is free.
\end{itemize}
\end{Rem}

\subsection{Decomposition matrix for index two subgroups}
\label{sub:decompositionindectwo}

    In this section, we keep the notation of the last section  and we
assume that $p$ is odd. The following result can be seen as an analogue of
Clifford theory for the projective indecomposable modules. 

\begin{Prop}
\label{prop:cliffpims}
Let $G$ be a finite group and $H$ be a subgroup of $G$ of index $2$. Let
$p$ be an odd prime number dividing $|G|$. Write $\varepsilon:G\rightarrow
\{-1,1\}$ for the surjective morphism with kernel $H$ induced by the
canonical projection $G\rightarrow G/H$. Then
\begin{equation}
\label{eq:permpim}
  \forall \varphi\in\IBr_p(G), \;  \varepsilon\otimes\Phi_\varphi=\Phi_{\varepsilon\otimes\varphi}.
\end{equation}
Moreover,
\begin{enumerate}
\item If $\varepsilon\otimes\varphi\neq\varphi$, then
$\Res_{H}^G(\Phi_\varphi) = \Res_H^G(\Phi_{\varepsilon\otimes\varphi})\in
\operatorname{IPr}_p(H)$. 
\item If $\varepsilon\otimes\varphi=\varphi$, then $\Phi_\varphi$ splits
into two projective indecomposable characters of $H$.
\end{enumerate}
All projective indecomposable characters of $H$ are obtained
by this process.
\end{Prop}

\begin{proof}
    First, we remark that $\varepsilon\otimes\varphi\in\IBr_p(G)$ for all
$\varphi\in\IBr_p(G)$. Furthermore, for $\varphi,\,\vartheta\in\IBr_p(G)$
we have
\begin{align*}
 \langle
\varepsilon\otimes\Phi_{\varphi},\varepsilon\otimes\vartheta\rangle_G&=\frac
1 {|G|}\sum_{g\in
G}\overline{\varepsilon(g)\Phi_\varphi(g)}\varepsilon(g)\vartheta(g)\\
&=\frac
1 {|G|}\sum_{g\in
G}\varepsilon(g)^2\overline{\Phi_\varphi(g)}\vartheta(g)\\
&=\langle\Phi_{\varphi},\vartheta\rangle_G\\
&=\delta_{\varphi,\vartheta}\\
&=\delta_{\varepsilon\otimes\varphi,\varepsilon\otimes\vartheta}\\
&=\langle\Phi_{\varepsilon\otimes\varphi},\varepsilon\otimes\vartheta\rangle_G.
\end{align*}
Hence, by uniqueness of the dual basis, we deduce that
$$\varepsilon\otimes\Phi_{\varphi}=\Phi_{\varepsilon\otimes\varphi}.$$
    Now, since $p$ is odd and $G/H$ is cyclic of order prime to $p$,
Clifford's theory for Brauer characters~\cite[Theorem 9.18]{hupp} can be
applied. We have
\begin{itemize}
\item If $\varepsilon\otimes\varphi\neq\varphi$, then
$\Res_H^G(\varphi)=\Res_H^G(\varepsilon\otimes\varphi)\in\IBr_p(H)$.  To
simplify the notation, we still denote by $\varphi$ the restriction of
$\varphi$ to $H$.  ~\cite[Theorem 9.8]{hupp} also gives that
\begin{equation} 
\label{eq:ind1}
\Ind_H^G(\varphi)=\varphi+\varepsilon\otimes\varphi. 
\end{equation}
\item If $\varepsilon\otimes\varphi=\varphi$, then
$\Res_H^G(\varphi)=\varphi^++\varphi^-$ with $\varphi^{\pm}\in\IBr_p(H)$,
and~\cite[Theorem 9.8]{hupp} again gives that 
\begin{equation}
\label{eq:ind2} 
\Ind_H^G(\varphi^+)=\Ind_H^G(\varphi^-)=\varphi.
\end{equation}
\end{itemize}
    Write $\operatorname{IPr}_p(H)=\{\Psi_\alpha\mid \alpha \in
\IBr_p(H)\}$. Let $\varphi \in \IBr_p(G)$. Since $\Res_H^G(\Phi_\varphi)$
vanishes on $p$-singular elements of $H$, it is a generalized projective
character of $H$. Thus, there are integers $a_{\alpha}$ for
$\alpha\in\IBr_p(H)$ such that 
$$\Res_H^G(\Phi_\varphi)=\sum_{\alpha\in\IBr_p(H)}a_\alpha
\Psi_\alpha.$$
Frobenius reciprocity gives that
\begin{align*}
a_{\alpha}=\langle \Res_H^G(\Phi_{\varphi}),\alpha\rangle_H
=\langle \Phi_{\varphi},\Ind_H^G(\alpha)\rangle_G,
\end{align*}
which is zero except for
$\alpha=\Res_{H}^G(\varphi)=\Res_{H}^G(\varepsilon\otimes\varphi)$ when
$\varepsilon\otimes\varphi\neq\varphi$ by~(\ref{eq:ind1}), and for
$\alpha=\varphi^\pm$ when $\varepsilon\otimes\varphi=\varphi$
by~(\ref{eq:ind2}). In these last cases, $a_{\alpha}=1$. Hence, we obtain
$$\Res_{H}^G(\Phi_{\varphi})=\Res_{H}^G(\varepsilon\otimes
\Phi_{\varphi})=\Psi_{\varphi}\quad\text{if
}\varepsilon\otimes\varphi\neq\varphi, 
$$
and
$$\Res_{H}^G(\Phi_{\varphi})=\Psi_{\varphi^+}+\Psi_{\varphi^-}\quad\text{if
}\varepsilon\otimes\varphi=\varphi, 
$$
as required. Finally, every projective indecomposable character of $H$ is
obtained by this process, because if $\Psi_{\alpha} \in
\operatorname{IPr}_p(H)$, then there is $\varphi\in\IBr_p(G)$ such that
$\langle \Psi_{\alpha},\Res_H^G(\Phi_\varphi)\rangle_H=\langle
\Ind_H^G(\Psi_{\alpha}),\Phi_\varphi\rangle_G\neq 0$ since
$\Ind_H^G(\Psi_{\alpha})\in\Z\operatorname{IPr}_p(G)$. The result follows.
\end{proof}

\begin{Rem}
    Note that, by Clifford theory, the constituents $\varphi^{\pm} \in
\IBr_p(H)$ of an $\varepsilon$-stable Brauer character
$\varphi\in\IBr_p(G)$ are $G$-conjugate, that is
$$\varphi^{+} = {}^{\sigma}\varphi^{-}$$ 
for some automorphism $\sigma\in\operatorname{Aut}(H)$ induced by an inner
automorphism of $G$. In particular, Lemma~\ref{lemma:automorphisme}
implies that
\begin{equation}
\label{eq:imageautopim}
{}^{\sigma}\Phi_{\varphi^+}=\Phi_{\varphi^-}.
\end{equation}
\end{Rem}

    By~(\ref{eq:defchapeau}), we have $\varepsilon \otimes
\widehat{\varphi} = \widehat{\varepsilon \otimes \varphi}$ for all class
functions $\varphi$ on $G$. Hence~(\ref{eq:relbasique}) gives
$d_{\varepsilon\otimes\chi,\varepsilon\otimes\varphi}=d_{\chi,\varphi}$
for any $\chi\in\Irr(G)$ and $\varphi\in\IBr_p(G)$. It follows that
$\varepsilon$ acts on the $p$-blocks of $G$. Let $\mathfrak B$ be a
$p$-block of $G$. We write $\Res_H^G(\mathfrak B)$ for the set of
constituents of the $\Res_H^G(\chi)$ for $\chi\in \mathfrak B$. The last
proposition gives information on the decomposition matrix of
$\Res_H^G(\mathfrak B)$ from that of $\mathfrak B$ as follows.

\begin{Prop}
\label{prop:restrictbloc}
Let $\mathfrak{B}$ be a $p$-block of $G$. Write $D_\mathfrak{B}$ for the
decomposition matrix of $\mathfrak{B}$.
\begin{enumerate}
\item If $\mathfrak{B}\neq\varepsilon(\mathfrak{B})$, then
$b=\Res_H^G(\mathfrak{B})=\Res_H^G(\varepsilon(\mathfrak{B}))$ is a
$p$-block of $H$, and the restriction of $D_\mathfrak{B}$ to $H$ (which is
equal to that of $D_{\varepsilon(\mathfrak{B})}$) is the decomposition
matrix of $b$.
\item If $\mathfrak{B}=\varepsilon(\mathfrak{B})$, then
$\Res_H^G(\mathfrak{B})$ is the sum of at most two $p$-blocks of $H$.
Moreover, if there is $\alpha\in \mathfrak{B}$ such that
$\alpha\neq\varepsilon\otimes\alpha$, then it is a single $p$-block whose
decomposition matrix has columns $\Psi_{\varphi}$ for $\varphi\in
\mathfrak{B}$ such that $\varphi\neq\varepsilon\otimes\varphi$ and
$\Psi_{\varphi^+}$ and $\Psi_{\varphi^-}$ otherwise.
\end{enumerate}
\end{Prop}

\begin{proof}
    Assume $\mathfrak{B}\neq\varepsilon(\mathfrak{B})$. Then $\mathfrak{B}
\cap \varepsilon(\mathfrak{B})=\emptyset$ and $\mathfrak{B}$ contains no
$\varepsilon$-stable character. The columns of $D_\mathfrak{B}$ and
$D_{\varepsilon(\mathfrak{B})}$ are the ordinary constituents of
$\Phi_{\varphi}$ and $\Phi_{\varepsilon\otimes\varphi}$ for $\varphi\in
\mathfrak{B}$. Thanks to~(\ref{eq:permpim}), the columns of
$D_\mathfrak{B}$ and $D_{\varepsilon(\mathfrak{B})}$ are interchanged by
$\varepsilon$, and part 1 of Proposition~\ref{prop:cliffpims} implies that
$\mathfrak b=\Res_H^G(\mathfrak{B})$ is a single $p$-block of $H$ and that
the restriction of $D_\mathfrak{B}$ to $H$ is the decomposition matrix of
$\mathfrak b$.

    Assume now $\mathfrak{B}=\varepsilon(\mathfrak{B})$. Let $\alpha$ and
$\beta$ be in $\mathfrak{B}$. Then there are $\alpha_0,\,\ldots,\alpha_r\in
\mathfrak{B}$ with $\alpha_0=\alpha$, $\alpha_r=\beta$ and such that
$\langle \Phi_{\alpha_i}, \Phi_{\alpha_{i+1}} \rangle_G \neq 0$ for $0\leq
i\leq r-1$. Furthermore, we derive from the proof
of~Proposition~\ref{prop:cliffpims} that $\Ind_H^G(\Psi_{\beta}) =
\Phi_{\beta}+\varepsilon \otimes \Phi_{\beta}$ if $\beta \neq
\varepsilon \otimes \beta$, and $\Ind_H^G(\Psi_{\beta^+}) =
\Ind_H^G(\Psi_{\beta^-}) = \Phi_\beta$ otherwise. Suppose that
$\alpha_i=\alpha_{i} \otimes \varepsilon$ and $\alpha_{i+1} = \alpha_{i+1}
\otimes \varepsilon$. Then
$$
\langle\Psi_{\alpha_i^+}+\Psi_{\alpha_i^-},\Psi_{\alpha_{i+1}^{\pm}}\rangle_H=\langle
\Res_H(\Phi_{\alpha_i}),\Psi_{\alpha_{i+1}^{\pm}}\rangle_H=
\langle\Phi_{\alpha_i},\Ind_H^G(\Psi_{\alpha_{i+1}^{\pm}})\rangle_G=
\langle\Phi_{\alpha_i},\Phi_{\alpha_{i+1}}\rangle_G\neq
0.$$
It follows that either $\langle \Psi_{\alpha_i^+} ,
\Psi_{\alpha_{i+1}^{\pm}} \rangle_H \neq 0$ or $\langle \Psi_{\alpha_i^-}
, \Psi_{\alpha_{i+1}^{\pm}} \rangle_H \neq 0$. Assume $\alpha_i \neq
\varepsilon\alpha_i$. If $\alpha_{i+1} = \varepsilon\alpha_{i+1}$, then
$$\langle\Psi_{\alpha_i},\Psi_{\alpha_{i+1}^{\pm}}\rangle_H=\langle
\Res_H^G(\Phi_{\alpha_i}),\Psi_{\alpha_{i+1}^{\pm}}\rangle_H=\langle
\Phi_{\alpha_i},\Ind(\Psi_{\alpha_{i+1}^{\pm}})\rangle_G
=\langle\Phi_{\alpha_i},\Phi_{\alpha_{i+1}}\rangle_G\neq
0.$$
If $\alpha_{i+1}\neq\varepsilon\alpha_{i+1}$, then
 $$\langle\Psi_{\alpha_i},\Psi_{\alpha_{i+1}}\rangle_H=\langle
\Res_H^G(\Phi_{\alpha_i}),\Psi_{\alpha_{i+1}}\rangle_H=\langle
\Phi_{\alpha_i},\Ind(\Psi_{\alpha_{i+1}})\rangle_G
=\langle\Phi_{\alpha_i},\Phi_{\alpha_{i+1}}+\varepsilon\Phi_{\alpha_{i+1}}\rangle_G\neq
0.$$
    Now, assume there is $\alpha\in \mathfrak{B}$ such that
$\alpha\neq\varepsilon\otimes\alpha$. Let $\beta\in \mathfrak{B}$. Then
there are $\alpha_1,\ldots,\alpha_r\in \mathfrak{B}$ as above, and the
previous computations show that the constituents of $\Res_H^G(\alpha_i)$
and of $\Res_H^G(\alpha_{i+1})$ for any $0\leq i \leq r-1$ are in the same
$p$-block of $H$. Hence, $\Res_H^G(\mathfrak B)$ is a $p$-block of $H$,
and the result follows. 
\end{proof}

\subsection{Relation between unitriangular $p$-basic sets of  $G$ and
$H$}

We now make the following assumption

\begin{Hyp}
\label{hyp:cadre}
Let $G$ be a finite group and $H$ be a normal subgroup of index $2$. Let
$x\in G\backslash H$. Denote by $\sigma$ the automorphism of
$H$ induced by conjugation by $x$, and $\varepsilon$ the linear
character of $G$ induced by the canonical morphism $G\rightarrow G/H$. Let
$p$ be an odd prime number and $\mathfrak b$ be a union of $p$-blocks of
$H$ covered by a union $\mathfrak B$ of $p$-blocks of $G$. 
\end{Hyp}

    Let $\mathfrak{B}$ be a $p$-block of $G$. For any $p$-basic set $B$ of
$\mathfrak B$, we denote by $\Res_H^G(B)$ the set of constituents of the
$\Res_H^G(\chi)$ for $\chi\in B$.

\begin{Prop}
\label{prop:nombre}
With the notation as above, if $B$ is an $\varepsilon$-stable $p$-basic
set of a union of $p$-blocks $\mathfrak B$ of $G$ then $$|\IBr_p(
\Res^G_H(\mathfrak B)|=|\Res_H^G(B)|.$$
\end{Prop}

\begin{proof}
    We follow step by step the proof of~\cite[Proposition 6.1]{BG}. This
is given for the symmetric group $\sym_n$, the sign character
$\varepsilon$ and the full decomposition matrix of $\Irr(\mathfrak S_n)$
but the argument is analogous for a group $G$ with index two subgroup
$H$ and a $p$-block $\mathfrak B$ with $\varepsilon$-stable $p$-basic set
$B$. We obtain
\begin{equation}
\label{eq:cardinalite}
|\operatorname{Fix}_{\varepsilon}(B)| = |\operatorname{Fix}_{\varepsilon}(
\IBr_p(\mathfrak B))|, 
\end{equation}
where $\operatorname{Fix}_{\varepsilon}(B))$ is the subset of
$\varepsilon$-fixed characters of $B$. 
    Now, since $p$ is odd, Clifford's modular theory~\cite[Theorem
9.18]{hupp} implies that each Brauer character of $\operatorname{ Fix}_{
\varepsilon}( \IBr( \mathfrak B))$ splits into two irreducible Brauer
characters of $\Res_H^G(\IBr(\mathfrak B))$, while the others restrict
irreducibly to $H$. Using that every irreducible character $\chi$ of $B$
also splits into two or one irreducible character(s) of $\Res_H^G(B)$
depending on whether $\chi \in \operatorname{ Fix}_{\varepsilon}( B)$ or
not, the result follows.
\end{proof}

\begin{Rem}
Let $B$ be an $\varepsilon$-stable $p$-basic set of a $p$-block $\mathfrak
B$ of $G$. Even though $\Res_H^G(B)$ is not a $p$-basic set of $\Res_H^G(
\mathfrak B)$ in general, Proposition~\ref{prop:nombre} asserts however
that we can derive from $B$ the number of Brauer characters of
$\Res_H^G(\mathfrak B)$.   
\end{Rem}

The following result is one of the main results of this paper. 
\begin{Th} 
\label{thm:basicGH}
We assume that Hypothesis~\ref{hyp:cadre} holds. We suppose that
$\mathfrak B$ has an $\varepsilon$-stable unitriangular $p$-basic set $(B,
\leq, \Theta)$ such that $\Theta:B \rightarrow \IBr_p(\mathfrak B)$ is
$\varepsilon$-equivariant. Then $b = \Res_H^G(B)$ is a unitriangular
$p$-basic set of $\mathfrak b$.
\end{Th}

\begin{proof}
    We consider a subset $A=\{\chi_1,\ldots,\chi_t\}$ of $B$ that
contains all $\varepsilon$-stable characters of $B$, and only one of
$\chi$ and $\varepsilon\otimes\chi$ when $\chi$ is a non
$\varepsilon$-stable character of $B$.
    By Clifford theory, each character of $A$ is above a character of $b$.
Furthermore, we suppose that the characters are chosen such that
$$\chi_1\leq\chi_2\leq\cdots\leq\chi_s\quad\text{and}\quad
\varepsilon\otimes\chi_i\leq\chi_i \ \text{if } i\ \text{is such that }
\varepsilon\otimes\chi_i\neq\chi_i.$$
    Since $B$ is $\varepsilon$-stable, Condition (ii) and
Equation~(\ref{eq:permpim}) give that $\chi_i$ is $\varepsilon$-stable if
and only if $\Phi_{\Theta(\chi_i)}$ is. If $\chi_i$ is $\varepsilon$-stable,
then we denote by $\varphi_i^{\pm}$ the constituents of
$\Res_H^G(\Theta(\chi_i))$. If $\chi_i\neq \varepsilon\otimes \chi_i$,
then $\Theta(\chi_i)\neq\varepsilon\otimes\Theta(\chi_i)$ by
Condition~(ii), and we write $\varphi_i=\Res_H^G(\Theta(\chi_i))\in
\IBr_p(\mathfrak b)$. 
    Now, we order $\Res_H^G(B)$ such that if $\psi$ and $\psi'$ are
constituents of $\Res_H^G(\chi_i)$ and $\Res_H^G(\chi_j)$ with $i\leq j$,
then $\psi\leq\psi'$. Suppose that $\chi_i\in B$ is $\varepsilon$-stable.
Write $\psi_i^+$ and $\psi_i^-$ for the constituents of
$\Res_H^G(\chi_i)$.
We have
$$\begin{array}{rcl}
d_{\psi_i^{\pm},\varphi_i^{+}}+d_{\psi_i^{\pm},\varphi_i^{-}}&=&\langle
\psi_i^{\pm},\Phi_{\varphi_i^{+}}\rangle_H + \langle
\psi_i^{\pm},\Phi_{\varphi_i^{\-}}\rangle_H\\
&=&\langle
\psi_i^{\pm},\Res_H^G(\Phi_{\Theta(\chi_i))}\rangle_H \\
&=&\langle\chi_i,\Phi_{\Theta(\chi_i)}
\rangle_G\\ 
&=&1.\end{array}$$
Hence $\{d_{\psi_i^{\pm}, \varphi_i^{+}}, d_{\psi_i^{\pm},
\varphi_i^{-}}\} = \{0,1\}$, because decomposition numbers are non-negative
integers. We assume that the labeling of $\psi_i^{\pm}$ is chosen such
that $d_{\psi_i^{+},\varphi_i^{+}}=1$ and
$d_{\psi_i^{+},\varphi_i^{-}}=0$. Furthermore,~(\ref{eq:coeffauto})
implies that $d_{\psi_i^{-},\varphi_i^{-}}=1$ and
$d_{\psi_i^{-},\varphi_i^{+}}=0$. With these choices, we  define
$\psi_i^+\leq\psi_i^{-}$.
    Finally, we set
$$\begin{array}{cccc}
\Psi:&\Res_H^G(B)&\longrightarrow &\IBr_p(\mathfrak b),\\
&\psi_i^{\pm}&\longmapsto &\varphi_i^{\pm}\\
&\psi_i&\longmapsto& \varphi_i
\end{array}.$$
Assume that $1\leq i\leq j\leq t$.
\begin{itemize}
\item Suppose that $\psi_i$ and $\psi_j$ are $\sigma$-stable.
Then
\begin{equation}
\label{eq:inter}
\begin{array}{rcl}
d_{\psi_i,\varphi_j}&=&\langle \psi_i,\Phi_{\varphi_j}\rangle_H\\
&=&
\langle \Ind_H^G(\psi_i),\Phi_{\Theta(\chi_j)}\rangle_G\\
&=&\langle
\chi_i+\varepsilon\otimes
\chi_i,\Phi_{\Theta(\chi_j)}\rangle_G\\
&=&d_{\chi_i,\Theta(\chi_j)}+d_{\varepsilon\otimes\chi_i,\Theta(\chi_j)}.
\end{array}
\end{equation}
However, $\varepsilon\otimes\chi_i\leq \chi_i\leq \chi_j$. If $i<j$ then
$d_{\chi_i,\Theta(\chi_j)}=0=d_{\varepsilon\otimes\chi_i,\Theta(\chi_j)}$,
and $d_{\psi_i,\varphi_j}=0$. If $i=j$, then~(\ref{eq:inter}) gives
$d_{\psi_i,\varphi_i}=1$ because
$d_{\chi_i,\Theta(\chi_i)}=1$, and
$d_{\varepsilon\otimes\chi_i,\Theta(\chi_i)}=0$ since
$\varepsilon\otimes\chi_i\leq \chi_i$.
\item Suppose that $\psi_i$ is $\sigma$-stable and $j$ labels $\psi_j^+$
and $\psi_j^-$.
Then
\begin{align*}
d_{\psi_i,\varphi_j^{+}}+d_{\psi_i,\varphi_j^-}&=
\langle \psi_i,\Phi_{\varphi_j^+}\rangle_H + \langle
\psi_i,\Phi_{\varphi_j^-}\rangle_H=
\langle \psi_i,\Phi_{\varphi_j^+}+\Phi_{\varphi_j^-}\rangle_H
=\langle \psi_i,\Res_H^G(\Phi_{\Theta(\chi_j)})\rangle_H\\
&=
\langle \Ind_H^G(\psi_i),\Phi_{\Theta(\chi_j)}\rangle_G
=\langle \varepsilon\otimes\chi_i+\chi_i,\Phi_{\Theta(\chi_j)}\rangle_G\\
&=d_{\chi_i,\Theta(\chi_j)}+d_{\varepsilon\otimes\chi_i,\Theta(\chi_j)}\\
&=0
\end{align*}
since $\varepsilon\otimes\chi_i\leq \chi_i < \chi_j$.
\item Suppose that $i$ and $j$ label non $\sigma$-stable characters.
Assume $i<j$. The same computation as above gives
\begin{align*}
d_{\psi_i^{\pm},\varphi_j^{+}}+d_{\psi_i^{\pm},\varphi_j^-}&=
\langle \chi_i,\Phi_{\Theta(\chi_j)}\rangle_G\\
&=d_{\chi_i,\Theta(\chi_j)}\\
&=0.
\end{align*}
If $i=j$, then $d_{\psi_i^+,\varphi_i^+}=1=d_{\psi_i^-,\varphi_i^-}$ and 
$d_{\psi_i^+,\varphi_i^-}=0$ by construction.
\item Suppose that $i$  labels two characters and $\chi_j$ is
$\varepsilon$-stable. Then
\begin{align*}
d_{\psi_i^{\pm},\varphi_j}&=
\langle \psi_i^{\pm},\Res_H^G(\Phi_{\Theta(\chi_j)})\rangle_H
=\langle \Ind_H^G(\psi_i^{\pm}),\Phi_{\Theta(\chi_j)}\rangle_G\\
&=\langle \chi_i,\Phi_{\Theta(\chi_j)}\rangle_G\\
&=d_{\chi_i,\Theta(\chi_j)}\\
&=0.
\end{align*}
\end{itemize}
This proves the result.
\end{proof}

\noindent We consider the set $\mathcal T$ of $\varepsilon$-stable
irreducible characters $\chi$ of $G$ such that the constituents $\chi^+$
and $\chi^-$ of their restriction to $H$ satisfies $\widehat{\chi}^+ \neq
\widehat{\chi}^-$.
\label{def:defT}

\begin{Rem}
\label{rk:ensembleTrestriction} 
Note that, if $B$ is a unitriangular $p$-basic set of $G$ which satisfies
the assumptions of Theorem~\ref{thm:basicGH}, then the $\sigma$-stable
characters of $B$ lie in $\mathcal T$. Indeed, since $\Res_H^G(B)$ is a
$p$-basic set of $H$, for any $\varepsilon$-stable character $\chi$ of
$B$, by Remark~\ref{rk:defbasic}, the family $(\widehat \chi^+,
\widehat\chi^-)$ is free and, in particular, $\widehat\chi^+ \neq
\widehat\chi^-$.
\end{Rem}

\begin{exa}
\label{ex:S6A6}
Let $G=\sym_6$ and $H=\alt_6$. Write $\operatorname{sgn}$ for the sign
character of $G$. The principal $3$-block $\mathfrak B_0$ of $G$ has $9$
irreducible characters and $5$ Brauer characters. It has a unitriangular
$3$-basic set $(B,\leq,\Theta)$ with $B=\{\chi_1, \chi_2, \chi_3, \chi_4,
\chi_5\}$, such that $\chi_2 = \operatorname{sgn} \otimes \chi_1$,
$\chi_4 = \operatorname{sgn} \otimes \chi_3$, and $\chi_5$ is
$\varepsilon$-stable and lies in $\mathcal T$. 
The restriction of the $3$-decomposition matrix to $B$ is

{\tikzstyle{every picture}+=[remember picture]
\renewcommand{\arraystretch}{1.2}
\renewcommand{\arraycolsep}{2pt}
$$
\begin{array}{c|p{1pt}cccccp{1pt}|}
\cline{2-2}
\cline{8-8}
\chi_1&  & 1 & 0 & 0 & 0 & 0 &  \\
\chi_2&  & 0 & 1 & 0 & 0 & 0 &  \\
\chi_3&  & 1 & 0 & 1 & 0 & 0 &  \\
\chi_4&  & 0 & 1 & 0 & 1 & 0 &  \\
\chi_5&  & 1 & 1 & 1 & 1 & 1 &  \\
\cline{2-2}
\cline{8-8}
\end{array}.
$$
}

    We remark that $\Theta$ is $\varepsilon$-equivariant. We set
$\widetilde{\varphi}_i=\Theta(\chi_i)$. Furthermore, we write $\psi_i =
\Res_H^G(\chi_i)$ for $1\leq i\leq 4$ and $\Res_H^G(\chi_5) =
\psi_5^++\psi_5^-$. Then the set $A$ appearing in the proof of
Theorem~\ref{thm:basicGH} is $A=\{\chi_2, \chi_4, \chi_5\}$. Set
$\varphi_2 = \Res_H^G(\widetilde{\varphi}_1)$, $\varphi_4 =
\Res_H^G(\widetilde{\varphi}_4)$ and $\Res_H^G(\widetilde{\varphi}_5) =
\varphi_5^++\varphi_5^-$. Using the fact that $\langle \psi_5^+,
\Phi_{\varphi_i}\rangle_H = \langle \psi_5^-, \Phi_{\varphi_i}\rangle_H =
d_{\chi_5,\Phi_{ \widetilde{ \varphi}_i}}$ for $i\in\{2,4\}$ (see for
example the last computation in the proof of Theorem~\ref{thm:basicGH}),
we deduce from Theorem~\ref{thm:basicGH} that the restriction to
$\{\psi_2,\psi_4,\psi_5^+,\psi_5^-\}$ of the $3$-decomposition matrix of
the principal block of $\alt_6$ is

{\tikzstyle{every picture}+=[remember picture]
\renewcommand{\arraystretch}{1.2}
\renewcommand{\arraycolsep}{2pt}
$$
\begin{array}{c|p{1pt}ccccp{1pt}|}
\cline{2-2}
\cline{7-7}
\psi_2&  & 1 & 0 & 0 & 0 &  \\
\psi_4&  & 1 & 1 & 0 & 0 &  \\
\psi_5^+&  & 1 & 1 & 1 & 0 &  \\
\psi_5^-&  & 1 & 1 & 0 & 1 &  \\
\cline{2-2}
\cline{7-7}
\end{array}.
$$
}
\end{exa}
\bigskip

    Now, we assume that Hypothesis~\ref{hyp:cadre} holds and that
$\mathfrak b$ has a unitriangular $p$-basic set $(b,\,\leq,\,\Psi)$. We
write $\mathcal R_{\mathfrak b}$ for the set of Brauer characters
$\widetilde{\varphi}\in\IBr_p(G)$ such that the restriction of $\varphi$
to $H$ splits into the sum of two Brauer characters of $\mathfrak b$. Set 
$$\mathcal S_{\mathfrak b}=\{\varphi,\,{}^{\sigma}\varphi\mid
\widetilde{\varphi}\in \mathcal
R_{\mathfrak b}\}\quad\text{and}\quad \mathcal
C_{b}=\{\Psi^{-1}(\varphi)\mid\varphi\in\mathcal S_{\mathfrak b}\}.$$

\begin{Not}
\label{notation:choixH}
Assume Hypothesis~\ref{hyp:cadre} is satisfied, and that $\mathfrak b$ has
a unitriangular $p$-basic set $(b,\leq,\Psi)$. For any $\varphi\in\mathcal
S_{\mathfrak b}$, we write $E(\varphi)$ and $E(\varphi)^{\star}$ for the
irreducible characters of $b$ such that
\begin{equation}
\label{eq:propchoix}
\{E(\varphi),E(\varphi)^{\star}\}=\{\Psi^{-1}(\varphi),
\Psi^{-1}({}^{\sigma}\varphi)\}\quad\text{and}\quad E(\varphi)\leq
E(\varphi)^{\star}.
\end{equation}
Note that this construction depends on the basic set $b$.
\end{Not}

\begin{Th}
\label{thm:proprietecoolH}
Assume Hypothesis~\ref{hyp:cadre} holds and that $\mathfrak b$ has a
unitriangular $p$-basic set $(b,\leq,\Psi)$. Then, for any $\varphi \in
\mathcal S_{\mathfrak b}$, we have
\begin{equation}
\label{eq:propcool}
{}^{\sigma}\widehat{E}(\varphi)\neq
\widehat{E}(\varphi)\quad\text{and}\quad
E(\varphi)\leq{}^{\sigma}E(\varphi).
\end{equation}
Furthermore, if we write
$$\chi=\Ind_H^G(E(\varphi))\quad\text{and}\quad
\vartheta=\Ind_H^G(\varphi),$$
then $\chi\in\Irr(G)$ and
$$d_{\chi,\vartheta}=1.$$
\end{Th}

\begin{proof}
    First, we remark that  ${}^{\sigma}\mathfrak b \cap \mathfrak b \neq
\emptyset$ since $\varphi$ and ${}^{\sigma}\varphi$ lie in $\mathfrak b$.
Hence, ${}^{\sigma}\mathfrak b = \mathfrak b$. Furthermore,
$\Phi_{\varphi}$ and $\Phi_{{}^{\sigma}\varphi}$ are distinct by
Lemma~\ref{lemma:automorphisme} since $\varphi \neq {}^{\sigma}\varphi$. 

    By assumption, $b$ is a unitriangular $p$-basic set of $\mathfrak b$
and $d_{E(\varphi),\varphi}=1$. Hence, Equation~(\ref{eq:coeffauto}) gives
$$d_{{}^{\sigma}E(\varphi),{}^{\sigma}\varphi}=1\neq 0,$$ 
and $E(\varphi)^{\star} \leq {}^{\sigma}E(\varphi)$, where
$E(\varphi)^{\star}$ is defined in Notation~\ref{notation:choixH}.
Then~(\ref{eq:propchoix}) implies that $E(\varphi) \leq
{}^{\sigma}E(\varphi)$. 
    On the other hand, by the choice of $E(\varphi)$, we also have
$d_{E(\varphi),{}^{\sigma}\varphi} = 0$. Hence
$\widehat{E}(\varphi)\neq{}^{\sigma}\widehat{E}(\varphi)$, otherwise
$d_{E(\varphi),{}^{\sigma}\varphi} = d_{{}^{\sigma}E(\varphi),
{}^{\sigma}\varphi}=1$. 
    In particular, $E(\varphi) \neq {}^{\sigma}E(\varphi)$, and $\chi$ is
an irreducible character of $G$ which satisfies $\chi = \varepsilon \otimes
\chi$ by Clifford theory. Write
$$\widehat{\chi}=\sum_{\beta\in \IBr_p(G)}d_{\chi,\beta}\,\beta.
$$
By Clifford theory, restricting this relation to $H$, we obtain
$$\widehat{E}(\varphi)+{}^{\sigma}\widehat{E}(\varphi)=\sum_{\beta =
\varepsilon\otimes\beta\in \IBr_p(G)}d_{\chi,\beta}\,(\beta^+ + \beta^-)
+\sum_{\beta\neq\varepsilon\otimes\beta}(d_{\chi,\beta}+
d_{\chi,\varepsilon\otimes\beta})\beta,
$$
and the following two identities:
\begin{align*}
\widehat{E}(\varphi)&=\sum_{\beta = \varepsilon\otimes\beta}
(d_{E(\varphi),\beta^+}\beta^++d_{E(\varphi),\beta^-}\beta^-)+\sum_{\beta
\neq\varepsilon \beta}d_{E(\varphi),\beta}\,\beta,
\\
{}^{\sigma}\widehat{E}(\varphi)&=\sum_{\beta = \varepsilon\otimes\beta}
(d_{{}^{\sigma}E(\varphi),\beta^+}\beta^++d_{{}^{\sigma}
E(\varphi),\beta^-}\beta^-)+\sum_{\beta
\neq\varepsilon \beta}d_{{}^{\sigma}E(\varphi),\beta}\,\beta.
\end{align*}
By uniqueness of the coefficients in a basis,  for all $\beta\in\IBr_p(G)$
such that $\beta = \varepsilon \otimes \beta$, we obtain
\begin{equation}
\label{eq:relres}
d_{\chi,\beta} =
d_{E(\varphi),\beta^+}+d_{{}^{\sigma}E(\varphi),\beta^+}\quad\text{and}
\quad
d_{\chi,\beta} =
d_{E(\varphi),\beta^-}+d_{{}^{\sigma}E(\varphi),\beta^-}.
\end{equation}
In particular, $d_{\chi,\vartheta} = d_{E(\varphi),\varphi} +
d_{{}^{\sigma}E(\varphi),\varphi}$. Assume that $d_{\chi,\vartheta} > 1$.
Then Equation~(\ref{eq:coeffauto}) gives
$$\begin{array}{rcl}
d_{E(\varphi),{}^{\sigma}\varphi} &=& d_{{}^{\sigma}E(\varphi),\varphi}\\
&=& d_{\chi,\vartheta}-d_{E(\varphi),\vartheta}\\
&=& d_{\chi,\vartheta}-1\\
&>&0,
\end{array}$$
which is a contradiction. The result follows.
\end{proof}
For any subsets $A\subseteq \Irr(H)$ and $B\subseteq \IBr_p(H)$, we write
$D_{A,B}$ for the restriction of the $p$-decomposition matrix of $H$ to
$A\times B$.

\begin{Prop}
\label{prop:unitriangulairerestriction}
Assume Hypothesis~\ref{hyp:cadre} holds, and that $\mathfrak b$ has a
unitriangular $p$-basic set $(b,\leq,\Psi)$ Let $D$ be the $p$-decomposition
matrix of $G$. Then there is a subset $\mathcal T_b$ of $\mathcal T$ 
so that $D_{\mathcal T_b,\mathcal R_{\mathfrak b }}$ is
unitriangularisable.
\end{Prop}

\begin{proof}
    We label $\mathcal R_{\mathfrak b} = \{\widetilde{\varphi}_1, \ldots
,\widetilde{\varphi}_r\}$ so that
$$E(\varphi_1)\leq E(\varphi_2) \leq \cdots \leq E(\varphi_r).$$
For any $1\leq i\leq r$, set
$$\chi_i=\Ind_H^G(E(\varphi_i)).$$
Then $\chi_i\in\mathcal T$ by Theorem~\ref{thm:proprietecoolH}. Set
\begin{equation}
\label{eq:ensembleTb}
\mathcal T_{b}=\{\chi_1,\ldots,\chi_r\}.
\end{equation}
    We remark that the elements of $\mathcal T_b$ are pairwise distinct.
Indeed, if $\chi_i = \chi_j$ for $i\neq j$, then $E(\varphi_i) =
{}^{\sigma}E(\varphi_j)$. However, Theorem~\ref{thm:proprietecoolH} gives
that $E(\varphi_i) \leq {}^{\sigma}E(\varphi_i) = E(\varphi_j)$ and
$E(\varphi_j) \leq {}^{\sigma}E(\varphi_j) = E(\varphi_i)$. Hence,
$E(\varphi_i) = E(\varphi_j) = {}^{\sigma}E(\varphi_i)$, which is a
contradiction. In particular, $|\mathcal T_b| = r =|\mathcal R_b|$.

    Now, we define an order on $\mathcal T_{b}$ by setting $\chi_i\leq
\chi_j$ whenever $i\leq j$, and we set
$$\Theta:\mathcal T_{b}\longrightarrow \mathcal R_{\mathfrak b},\
\chi_i\longmapsto \widetilde{\varphi}_i.$$
The map $\Theta$ is well-defined and surjective by construction. Hence, it
is bijective by cardinality. Now, we prove that $D_{\mathcal T_{b},
\mathcal R_{\mathfrak b}}$ is unitriangular with respect to the datum
$(\leq,\Theta)$.
First, by Theorem~\ref{thm:proprietecoolH}, we have, for all $1\leq i\leq
r$
$$d_{\chi_i,\Theta(\chi_i)}=d_{\Ind_H^G(E(\varphi_i)),\Ind_H^G(\varphi_i)}=1.$$
Assume $i\leq j$ (in particular $\chi_i\leq \chi_j)$. Then
\begin{align*}
d_{\chi_i,\widetilde{\varphi}_j}&=\langle
\chi_i,\Phi_{\widetilde{\varphi_j}}\rangle_G\\
&=\langle
E(\varphi_i),\Res_H^G(\Phi_{\widetilde{\varphi}_j})\rangle_H\quad\text{(by
Frobenius Reciprocity)}\\
&=\langle
E(\varphi_i),\Phi_{\varphi_j}+\Phi_{{}^{\sigma}\varphi_j}\rangle_H\quad\text{(by
Proposition~\ref{prop:cliffpims})}\\
&=\langle E(\varphi_i),\Phi_{\varphi_j})\rangle_H +\langle
E(\varphi_i),\Phi_{{}^{\sigma}\varphi_j}\rangle_H\\
&=\langle E(\varphi_i),\Phi_{\Psi(E(\varphi_j))})\rangle_H +\langle
E(\varphi_i),\Phi_{{}^{\sigma}\Psi(E(\varphi_j))}\rangle_H\\
&=d_{E(\varphi_i),\Psi(E(\varphi_j))}+d_{E(\varphi_i),\Psi(E(\varphi_j)^*)}.
\end{align*}
However, $E(\varphi_i)\leq E(\varphi_j)\leq E(\varphi_j)^*$, hence
$$d_{E(\varphi_i),\Psi(E(\varphi_j))}=0\quad\text{and}\quad
d_{E(\varphi_i),\Psi(E(\varphi_j)^*)}=0.$$
    This proves that $d_{\chi_i,\widetilde{\varphi}_j}=0$ whenever
$\chi_i\leq \chi_j$, as required.
\end{proof}

\begin{Th}
\label{thm:basicHG}
Let $G$ be a finite group and $H$ be an index $2$ subgroup of $G$. Let $p$
be an odd prime number. Assume Hypothesis~\ref{hyp:cadre} is satisfied,
and that $\mathfrak b$ has a unitriangular $p$-basic set $(b,\leq,\Psi)$.
Let $\mathfrak B$ be the union of all $p$-blocks of $G$ covering $\mathfrak b$.
Then $\mathfrak B$ has a unitriangular $p$-basic set.
\end{Th}

\begin{proof}
    We keep Notation~\ref{notation:choixH}, and define $E_b =
\Irr(\mathfrak b) \backslash \{E(\varphi)^* \mid \widetilde{\varphi} \in
\mathcal R_{\mathfrak b}\}$. Suppose that $E_b = \{\psi_1,\ldots,\psi_s\}$
with
$$\psi_1\leq\psi_2\leq\cdots \leq \psi_s.$$
Note that, for all $1 \leq i\leq s$, we have ${}^{\sigma}\psi_i = \psi_i$
if and only if $\Phi_{\Psi(\psi_i)} = {}^{\sigma}\Phi_{\Psi(\psi_i)}$. For
any $1\leq i\leq s$, we write $\widetilde{\varphi}_i^+$ and
$\widetilde{\varphi}_i^-$ for the constituents of
$\Ind_{H}^{G}(\Psi(\psi_i))$ if $\psi_i$ is $\sigma$-stable, and
$\widetilde{\varphi}_i = \Ind_H^G(\Phi(\psi_i))$ if $\psi_i$ is not
$\sigma$-stable. Moreover, we also write
$$\Res_H^G(\widetilde{\varphi_i}^+)=\varphi_i\quad\text{and}\quad\Res_H^G(\widetilde{\varphi_i})=\varphi_i+{}^{\sigma}\varphi_i.$$
Assume that $\Ind_{H}^G(\psi_i)$ has two constituents, $\chi_i^+$ and
$\chi_i^-$. Then
\begin{equation}
\label{eq:labeling}
\begin{array}{rcl}
 d_{\chi_i^+,\widetilde{\varphi}_i^{\pm}}^{\pm}+d_{\chi_i^-,\widetilde{\varphi}_i^{\pm}}^{\pm}&=&
 \langle
 \Ind_H^G(\psi_i),\Phi_{\widetilde{\varphi}_i^{\pm}}\rangle_G\\
 &=&\langle
 \psi_i,\Res_H^G(\Phi_{\widetilde{\varphi}_i^{\pm}})\rangle_H\\
 &=&\langle
 \psi_i,\Phi_{\varphi_i}\rangle_H\\
 &=&1.
 \end{array}
\end{equation}
Since these are non-negative integers, one has to be equal to $1$ and the
other to $0$. We then choose the labeling such that $d_{\chi_i^+,
\widetilde{\varphi}_i^+} = 1$ and $d_{\chi_i^-, \widetilde{\varphi}_i^-} =
0$. Furthermore, since the characters of $\{\chi_i^+, \chi_i^-\}$ and of the
set  $\{\Phi_{\widetilde{\varphi}_i^+}, \Phi_{\widetilde{\varphi}_i^-}\}$
are interchanged by $\varepsilon$, we deduce that $d_{\chi_i^+,
\widetilde{\varphi}_i^-} = 0$ and $d_{\chi_i^-, \widetilde{\varphi}_i^-} =
1$.

    We define $B$ as the set of constituents of the $\Ind_H^G(\psi_i)$'s for $1
\leq i \leq s$. We order $B$ so that the natural order of the indices is 
respected and $\chi_i^+ \leq \chi_i^-$. Finally, we set
\begin{equation}
\label{eq:defthetaHG}
\begin{array}{cccc}
\Theta:&B&\longrightarrow &\IBr_p(\mathfrak B),\\
&\chi_i^{\pm}&\longmapsto& \widetilde{\varphi}_i^{\pm}\\
&\chi_i&\longmapsto &\widetilde{\varphi}_i
\end{array}.
\end{equation}
    Now, we prove that $B$ is a unitriangular $p$-basic set of $\mathfrak
B$ with respect to $(\leq,\Theta)$. We already checked  in
Proposition~\ref{prop:unitriangulairerestriction} that if $\chi_i$ and
$\chi_j$ are $\varepsilon$-stable and such that $\chi_i\leq\chi_j$, then
$d_{\chi_i,\Theta(\chi_i)}=1$ and $d_{\chi_i,\Theta(\chi_j)}=0$. 
Suppose $1\leq i< j\leq s$.
\begin{itemize}
\item Assume that $i$ labels two non
$\varepsilon$-stable characters $\chi_i^+$ and $\chi_i^-$. If $\chi_j$ is
$\varepsilon$-stable, then so is $\widetilde{\varphi}_j$ and
\begin{align*}
 d_{\chi_i^+,\widetilde{\varphi}_j}+d_{\chi_i^-,\widetilde{\varphi}_j}&=
 \langle \Ind_H^G(\psi_i),\Phi_{\widetilde{\varphi}_j}\rangle_G\\&=\langle
 \psi_i,\Res_H^G(\Phi_{\widetilde{\varphi}_j})\rangle_H\\&=\langle
 \psi_i,\Phi_{\varphi_j}\rangle_H+\langle
 \psi_i,\Phi_{{}^{\sigma}\varphi_j}\rangle_H\\&=0,
\end{align*}
because $\psi_j=E(\varphi_j)$, whence $\psi_i\leq E(\varphi_j)\leq E(\varphi_j)^*$.
It follows that
$d_{\chi_i^+,\widetilde{\varphi}_j}=0=d_{\chi_i^-,\widetilde{\varphi}_j}$
since both are non-negative integers.
If $j$ labels two non $\varepsilon$-stable characters, then $\Res_H^G
\Phi_{\widetilde{\varphi}_j^+}=\Res_H^G
\Phi_{\widetilde{\varphi}_j^-}=\Phi_{\varphi_j}$. Hence, an analogue
computation as above shows that
$$\begin{array}{rcl}
 d_{\chi_i^+,\widetilde{\varphi}_j^{\pm}}^{\pm}+d_{\chi_i^-,\widetilde{\varphi}_j^{\pm}}^{\pm}&=&
 \langle
 \Ind_H^G(\psi_i),\Phi_{\widetilde{\varphi}_j^{\pm}}\rangle_G\\
 &=&\langle
 \psi_i,\Res_H^G(\Phi_{\widetilde{\varphi}_j^{\pm}})\rangle_H\\
 &=&\langle
 \psi_i,\Phi_{\varphi_j}\rangle_H\\
 &=&0
 \end{array}
 $$
because $\psi_i\leq\psi_j$.
Finally, by the choice of our labeling (see the discussion
after~(\ref{eq:labeling})), we have
$d_{\chi_i^{\pm},\Phi_{\widetilde{\varphi}_i^{\pm}}}=1$ and
$d_{\chi_i^+,\Phi_{\widetilde{\varphi}_i^-}}=0$.
\item Assume $\chi_i$ is $\varepsilon$-stable and $j$ labels two non
$\varepsilon$-stable characters. Then
$$\begin{array}{rcl}
d_{\chi_i,\widetilde{\varphi}_j^{\pm}}&=&\langle
\chi_i,\Phi_{\widetilde{\varphi}_j^{\pm}}\rangle_G\\
&=&
\langle
\Ind_H^G(\psi_i),\Phi_{\widetilde{\varphi}_j^{\pm}}\rangle_G\\
&=&
\langle \psi_i,\Res_H^G(\Phi_{\widetilde{\varphi}_j^{\pm}})\rangle_H\\
&=&
\langle \psi_i,\Phi_{\varphi_j}\rangle_H=d_{\psi_i,\varphi_j}\\
&=&0.\end{array}$$ 
\end{itemize}
This proves the result.
\end{proof}

\begin{exa}
\label{ex:A6S6}
    Consider the unitriangular $3$-basic set $(b,\leq,\Psi)$ of the
principal block $\mathfrak b_{0}$ of $\alt_6$ obtained in
Example~\ref{ex:S6A6}, where $b=\{\psi_2,\psi_4,\psi_5^+,\psi_5^-\}$. We
have $\mathcal R_{\mathfrak b_{0}}=\{\chi_5\}$ and 
$$E(\varphi_5^+)=E(\varphi_5^-)=\psi_5^+\quad\text{and}\quad 
E(\varphi_5^+)^*=E(\varphi_5^-)^*=\psi_5^-.$$
Applying Theorem~\ref{thm:basicHG}, we obtain a unitriangular $3$-basic
set of $\mathfrak B_{0}$ with decomposition matrix of the form

{\tikzstyle{every picture}+=[remember picture]
\renewcommand{\arraystretch}{1.2}
\renewcommand{\arraycolsep}{2pt}
$$
\begin{array}{c|p{1pt}cccccp{1pt}|}
\cline{2-2}
\cline{8-8}
&  & 1 & 0 & 0 & 0 & 0 &  \\
&  & 0 & 1 & 0 & 0 & 0 &  \\
&  & * & * & 1 & 0 & 0 &  \\
&  & * & * & 0 & 1 & 0 &  \\
&  & 1 & 1 & 1 & 1 & 1 &  \\
\cline{2-2}
\cline{8-8}
\end{array}.
$$
}

Note that, even though we deduce the existence of a unitriangular $3$-basic set of $\mathfrak
B_{0}$, Theorem~\ref{thm:basicHG} does not give its complete
associated decomposition matrix. Since $d_{\psi_4,\Phi_{\varphi_2}}=1$, we
only know that the missing
values are either $\begin{pmatrix}
1&0\\
0&1
\end{pmatrix}$ or $\begin{pmatrix}
0&1\\
1&0
\end{pmatrix}$. 
\end{exa}

\begin{Rem} 
\label{rk:equivalence}
    The $p$-basic set $\widetilde{b}$ of $\mathfrak B$ constructed from $b$ in
Theorem~\ref{thm:basicHG} is $\varepsilon$-stable, and its
$\varepsilon$-stable characters lie in $\mathcal T$ by construction.
Furthermore, the bijection $\Theta$ is $\varepsilon$-equivariant
by~(\ref{eq:defthetaHG}). 
    In particular, $\widetilde{b}$ satisfies the assumptions of
Theorem~\ref{thm:basicGH}. Hence, if $\mathfrak b$ has a unitriangular
$p$-basic set, then $\mathfrak b$ has a unitriangular $p$-basic set that
comes from the restriction of an $\varepsilon$-stable $p$-basic set of
$\mathfrak B$. However, note that in general 
$$\Res_H^G(\widetilde{b})\neq b.$$
\end{Rem}

\begin{exa}
\label{ex:passtable}
    The principal $3$-block $\mathfrak b_{0}$ of $\alt_6$ has an
$\varepsilon$-stable character $\psi_6$ of degree $5$ whose
modular restriction with respect to the labeling of $\IBr_3(\mathfrak
b_{0})$ given in Example~\ref{ex:S6A6} is 
 $$[0\ 1\ 1\ 1].$$
    In particular, $b'=\{\psi_2,\psi_4,\psi_5^+,\psi_6\}$ is also a
unitriangular $3$-basic set of $\mathfrak b_{0}$. Note that, when
we apply Theorem~\ref{thm:basicHG} with $b'$, we obtain the same
unitriangular basic set of $\mathfrak B_0$ as that of
Example~\ref{ex:A6S6}.
\end{exa}

\section{Consequences for the $p$-basic sets of the alternating groups}
\label{sec:alterne}

\subsection{Ordinary and modular representations}

    Let $p$ be a prime number and $n$ be a positive integer. For this
section, we refer to~\cite{JK} for more details. For any
$\lambda\in\Pi^1(n)$ we denote by $\chi_{\lambda}$ the character of the
simple module $S^{\lambda}$ given in~(\ref{eq:sym}). We write $\lambda'$
for the conjugate partition of $\lambda$ and recall that $S^{\lambda'} =
\operatorname{sgn} \otimes S^{\lambda}$, where $\operatorname{sgn}$ is the
sign character of $\sym_n$. Furthermore, the irreducible characters of
$\alt_n$ can be described from those of $\sym_n$ as follows. If $\lambda \neq
\lambda'$, then $\Res_{\alt_n}^{\sym_n}(\chi_{\lambda}) =
\Res_{\alt_n}^{\sym_n}(\operatorname{sgn} \otimes \chi_{\lambda})$ is
irreducible and denoted by $\rho_{\lambda}$. If $\lambda = \lambda'$, then
$\Res_{\alt_n}^{\sym_n}(\chi_{\lambda})$ splits into two irreducible
characters $\rho_{\lambda}^+$ and $\rho_{\lambda}^-$ of $\alt_n$. These two
class functions take the same values everywhere, except on the elements with cycle
structure given by the partition $\overline{\lambda}$, whose parts are the
diagonal hook lengths of $\lambda$. These elements form a single
conjugacy class in $\sym_n$ which splits into two classes
$\overline\lambda_-$ and $\overline\lambda_+$ of $\alt_n$, and,
following~\cite[Theorem 2.5.13]{JK}, the notation can be chosen such that 
$$\rho_{\lambda}^{\pm}(t_{\overline\lambda_+}) = x_{\lambda}\pm
y_{\lambda}\quad \text{and} \quad
\rho_{\lambda}^{\pm}(t_{\overline\lambda_-}) = x_{\lambda}\mp
y_{\lambda},$$
where $t_{\overline\lambda_{+}}$ (resp. $t_{\overline\lambda_{-}}$) is a representative of the class
$\overline\lambda_{+}$ (resp. $\overline\lambda_{-}$) of $\alt_n$, and, if $\overline\lambda =
(d_1,d_2,\ldots,d_k)$, then
\begin{equation}
\label{eq:valAn}
x_{\lambda}=\frac 1 2(-1)^{(n-k)/2}\quad\text{and}\quad y_{\lambda}=
\frac 1 2\sqrt{(-1)^{(n-k)/2}d_1\cdots d_k}.
\end{equation}

\begin{lemma}
\label{lemme:TSn}
Let $p$ be an odd prime number. The set $\mathcal T$ defined on
page~\pageref{def:defT} for $\alt_n$ is 
$$\mathcal T=\{\chi_{\lambda}\mid \lambda\in\mathcal G\},$$
where $\mathcal G$ is the set of self-conjugate partitions none of whose
diagonal hooks has length divisible by $p$.
\end{lemma}

\begin{proof}
    For any self-conjugate partition $\lambda$, if $p$ divides a part of
$\overline \lambda$, then the corresponding class is $p$-singular. In
particular, $\widehat {\rho}_{\lambda}^{\pm}(t_{\overline\lambda_-}) = 0 =
\widehat {\rho}_{\lambda}^{\pm}(t_{\overline\lambda_+})$, and
$\widehat{\rho}_{\lambda}^+ = \widehat{\rho}_{\lambda}^-$. Thus,
$\chi_{\lambda} \notin \mathcal T$. Now, if $\lambda\in\mathcal G$, then
$\overline\lambda_+$ and $\overline\lambda_-$ label $p$-regular classes of
$\alt_n$. The values of $\widehat{\rho}^{+}_{\lambda}$ and
$\widehat{\rho}^{-}_{\lambda}$ on these classes are distinct
by~(\ref{eq:valAn}), and the result follows.
\end{proof}

    By \cite{JK}, the simple $\sym_n$-modules in characteristic $p$
are labeled by the set of $p$-regular partitions $\mathcal R_n^p$ of $n$.
Write $D^{\mu}$ for the simple $\sym_n$-module labeled by
$\mu\in\mathcal R_n^p$ and $\varphi_{\mu}$ for its Brauer character. For
any $p$-regular partition $\mu$, the Mullineux partition $m(\mu)$
associated to $\mu$ is the $p$-regular partition such that
$$\operatorname{sgn}\otimes D^{\mu}\simeq D^{m(\mu)}.$$ 
    We denote by $\mathcal{M}_n^p$ the set of $p$-regular partitions $\mu$
such that $m(\mu) = \mu$. Since $p$ is odd, and following the notation of
\S\ref{sub:decompositionindectwo}, we then have
\begin{equation}
\label{eq:brauerAn}
\IBr_p(\alt_n)=\{\varphi_{\mu}\mid \mu\neq
m(\mu)\}\cup\{\varphi_{\mu}^{\pm}\mid \mu\in\mathcal M_n^p\}. 
\end{equation}
    We write $\sigma$ for conjugation by the transposition $(1\ 2)$.
In particular, for any self-conjugate partition $\lambda\in\Pi_n^1$ and
$\mu\in \mathcal M_n^p$,
$${}^{\sigma}\rho_{\lambda}^{\pm}=\rho_{\lambda}^{\mp} \quad \text{and}
\quad {}^{\sigma}\varphi_{\mu}^{\pm}=\varphi_{\mu}^{\mp}.$$

    Now, we recall the construction of the $p$-blocks of $\sym_n$ and
$\alt_n$. To any partition $\lambda$ of $n$, we can associate its $p$-core
$\lambda_{(p)}$ and its $p$-quotient $\lambda^{(p)} = (\lambda^1, \ldots,
\lambda^p) \in \Pi^p(w)$, where $\Pi^p(n)$ is the set of
$p$-multipartitions of $n$ and $w=(n-|\lambda_{(p)}|)/p$. Note that the
map
\begin{equation}
\label{eq:corequotient}
\lambda\mapsto (\lambda_{(p)},\lambda^{(p)})
\end{equation} is bijective.
Recall that (see for example~\cite[Lemma 3.1]{BG})
\begin{equation}
\label{eq:conjcorequotient}
(\lambda')_{(p)}=\lambda_{(p)}'\quad \text{and}\quad 
(\lambda')^{(p)}=\left((\lambda^p)',(\lambda^{p-1})',\ldots,(\lambda^1)'\right).
\end{equation}
    Furthermore, the Nakayama Conjecture~\cite[\S 6.1]{JK} asserts that two
irreducible characters $\chi_{\lambda}$ and $\chi_{\mu}$ lie in the same
$p$-blocks of $\sym_n$ if and only if $\lambda_{(p)} = \mu_{(p)}$. It
follows that the $p$-block of $\sym_n$ are parametrized by the $p$-cores
of partitions of $n$. In particular, by~(\ref{eq:corequotient}), the
irreducible characters lying in a $p$-block of $\sym_n$ corresponding to a
$p$-core $\gamma$ are labeled by the set $\Pi^p(w)$, where $w =
(n-|\gamma|)/p$ is called the \emph{$p$-weight} of the block.

    The $p$-blocks of $\alt_n$ can then be parametrized as follows. Let
$\gamma$ be the $p$-core of a partition of $n$. Write
$\mathfrak{B}_{\gamma}$ for the $p$-block of weight $w$ of $\sym_n$ corresponding to
$\gamma$, and $\mathfrak b_{\gamma} = \Res_{\alt_n}^{\sym_n}(\mathfrak
B_{\gamma})$. If $\gamma\neq\gamma'$ then $\mathfrak b_{\gamma} = \mathfrak b_{\gamma'}$ is a
$p$-block of $\alt_n$. If $\gamma = \gamma'$ and $w>0$, then the character
$\chi_{\lambda}$ of $\sym_n$ such that $\lambda_{(p)} = \gamma$ and
$\lambda^{(p)} = ((w),\emptyset,\ldots,\emptyset)$ is not
$\varepsilon$-stable by~(\ref{eq:conjcorequotient}). Hence $\mathfrak b_{\gamma}$ is
again a single $p$-block of $\alt_n$ by Proposition~\ref{prop:restrictbloc}.
Finally, when $w = 0$, the block $\mathfrak{B}_\gamma$ contains a unique
character whose restriction to $\alt_n$ splits into two irreducible
characters of $\alt_n$, which give two $p$-blocks of $\alt_n$ with defect
zero.\medskip

\begin{Rem}
The set $\mathcal T$ of Lemma~\ref{lemme:TSn} is described in~\cite[Lemma
3.4]{BG} in terms of the bijection~(\ref{eq:corequotient}). It is the set
$\mathcal C$ of self-conjugate partitions $\lambda\in\Pi_n^1$ whose $p$-quotient's $(p+1)/2$-th part is the empty partition (or,
equivalently, whose diagonal hook lengths are prime to $p$).
\end{Rem}

\begin{lemma}
\label{lemme:CSn}
Let $p$ be an odd prime number. Let $\gamma$ be a symmetric $p$-core of
$n$ labeling a $p$-block of $\sym_n$ with positive $p$-weight. 
\begin{enumerate}[(i)]
\item If $\mathfrak b_{\gamma}$ has a $p$-basic set $b$, then
the non $\sigma$-stable characters of $b$ are labeled by the set 
$\mathcal C_{\gamma}$ of partitions of $\mathcal C$ with $p$-core
$\gamma$. Moreover, any $\lambda\in\mathcal C_{\gamma}$ labels at least one non
$\sigma$-stable character of $b$.
\item If $\mathfrak B_{\gamma}$ has an $\varepsilon$-stable  
$p$-basic set $B$ which
restricts to a $p$-basic set of $\mathfrak b_{\gamma}$, then 
the set of $\varepsilon$-stable characters of $B$ is labeled by
$\mathcal C_{\gamma}$.
\end{enumerate}
\end{lemma}

\begin{proof}
    Let $b$ be a $p$-basic set of $\mathfrak b_{\gamma}$. The non
$\sigma$-stable characters of $b$ have to be labeled by partitions of
$\mathcal C_{\gamma}$ by Lemma~\ref{lemme:TSn}. Let $\lambda\in\mathcal
C_{\gamma}$. Assume that $b$ contains neither $\rho_{\lambda}^+$ nor
$\rho_{\lambda}^-$.
    Since $b$ is a $p$-basic set of $\mathfrak b_{\gamma}$, by Remark~\ref{rk:defbasic}
there are $\psi_1,\ldots,\psi_r\in b$ and $a_1,\ldots,a_r\in\Z$ such that
\begin{equation}
\label{eq:basican}
\widehat{\rho}_{\lambda}^+ = \sum_{j = 1}^r a_j\widehat{\psi}_j.
\end{equation}
By Lemma~\ref{lemme:TSn}, $t_{\overline \lambda_{+}}$ and $t_{\overline \lambda_{-}}$ are $p$-regular elements.
    On the other hand, for any symmetric partition $\mu\neq\lambda$,
$\rho_{\mu}^{+}(t_{\overline \lambda_{\pm}}) = \rho_{\mu}^{-}(t_{\overline
\lambda_{\pm}})$. Now, using that $\rho_{ \lambda}^{\pm}\notin b$
and evaluating equality~(\ref{eq:basican}) on $t_{\overline
\lambda_{\pm}}$, we obtain
$$
\widehat{\rho}_{\lambda}^+(t_{\overline{\lambda}_{+}}) = \sum_{j=1}^r
a_j\widehat{\psi}_j(t_{\overline \lambda_+}) = 
\sum_{j=1}^r a_j\widehat{\psi}_j(t_{\overline \lambda_-}) = 
\widehat{\rho}_{\lambda}^+(t_{\overline{\lambda}_{-}}),$$
which is a contradiction. Hence, either $\rho_{\lambda}^+$ or
$\rho_{\lambda}^-$ lies in $b$.

    Suppose now there is an $\varepsilon$-stable $p$-basic set $B$ of
$\mathfrak B_{\gamma}$ such that $\Res^{\sym_n}_{\alt_n}(B) = b$. Again by
Lemma~\ref{lemme:TSn}, the $\varepsilon$-stable characters of $B$ are in
$\mathcal T$, that is, are labeled by partitions of $\mathcal C_{\gamma}$.
But at least one character of $b$ is labeled by a partition of $\mathcal
C_{\gamma}$ and its induced character to $\sym_n$ is in $B$. The result
follows.
\end{proof}

\begin{Rem}
\label{rk:ensTbAn}
\begin{enumerate}[(i)]
\item Note that if $b$ is the restriction of an $\varepsilon$-stable $p$-basic
set of $\mathfrak B_{\gamma}$, then $b$ contains both $\rho_{\lambda}^+$
and $\rho_{\lambda}^-$ for all $\lambda\in \mathcal C_{\gamma}$. The
converse is not true, as we see in Remark~\ref{rk:defbasic} and in
Example~\ref{ex:passtable}.
\item Assume $\mathfrak b_{\gamma}$ has a unitriangular $p$-basic set.
Then for any $\widetilde{\varphi}\in\mathcal R_{\mathfrak b}$, there
exists a unique $\lambda\in\mathcal C_{\gamma}$ such that either
$E(\varphi) = \rho_{\lambda}^+$ or $E(\varphi) = \rho_{\lambda}^-$. In
particular, the set $\mathcal T_b$ of
Proposition~\ref{prop:unitriangulairerestriction} constructed
in~(\ref{eq:ensembleTb}) is labeled by $\mathcal C_{\gamma}$.
\end{enumerate}
\end{Rem}

\begin{exa}
    We return to Example~\ref{ex:S6A6}. We have $\gamma = \emptyset$,
$\mathfrak B_{0}$ has $3$-weight $2$, and the characters
$\chi_1,\ldots,\chi_5$ are labeled by the partitions $(6)$, $(1^6)$,
$(5,1)$, $(2,1^4)$ and $(3,2,1)$, respectively. The $3$-quotient of
$(3,2,1)$ is $((1),\emptyset,(1))$. In particular, $(3,2,1)$ lies in
$\mathcal T$, and is the only self-conjugate partition of $6$ of weight $2$ in $\mathcal
T$.
\end{exa}

\subsection{Counterexamples for alternating groups}
\label{subsec:contreexemples}

    Now, we will prove that, for $p = 3$, the alternating groups $\alt_{18}$
and $\alt_{19}$ have no unitriangular $3$-basic set.

\subsubsection{The case of $\alt_{18}$ and $p = 3$}\label{exA18}

    In this section, we consider the principal block $\mathfrak
b_{0}$ of $\alt_{18}$ for $p = 3$. Assume $\mathfrak b_{0}$
has a unitriangular $3$-basic set $(b,\leq,\Psi)$. Let $\mathfrak
B_{0}$ be the principal block of $\sym_{18}$. Note that this is
the unique $3$-block of $\sym_{18}$ that covers $\mathfrak b_{0}$.
Furthermore, $\mathfrak B_{0}$ has exactly three PIMs
$\Phi_{\mu^1}$, $\Phi_{\mu^2}$ and $\Phi_{\mu^3}$ fixed under
$\operatorname{sgn}$, where $\mu^1$, $\mu^2$ and $\mu^3$ are the
$3$-regular partitions of $\mathcal M_{18}^3$ given by
$$\mu^1 = (10,4^2),\quad \mu^2 = (9,4^2,1)\quad\text{and}\quad
\mu^3 = (7,5,2^2,1^2).$$
    Since $\mathfrak b_{0}$ has a unitriangular $p$-basic set, by
Lemma~\ref{lemme:CSn} and Remark~\ref{rk:ensTbAn}, the set $\mathcal T_b$
defined in~(\ref{eq:ensembleTb}) is the set of characters of $\sym_{18}$
labeled by $\mathcal C_{0} =
\{\lambda^1,\,\lambda^2,\,\lambda^3\}$, where
$$
\lambda^1 = (9,2,1^7),\quad \lambda^2 = (7,4,2^2,1^3)\quad\text{and}\quad
\lambda^3 = (6,5,2^3,1).
$$
Note that $(\lambda^1)^{(3)} = ((1^3),\emptyset,(3))$,
$(\lambda^2)^{(3)} = ((3),\emptyset,(1^3))$ and
$(\lambda^3)^{(3)} = ((2,1),\emptyset,(2,1))$.

    Now, using Chevie~\cite{GAP4}, we obtain that the part of the
decomposition matrix of $\mathfrak B_{0}$ corresponding to these
characters is

{\tikzstyle{every picture}+=[remember picture]
\renewcommand{\arraystretch}{1.2}
\renewcommand{\arraycolsep}{2pt}
\begin{equation}
\label{eq:matrice18}
D=\begin{array}{c|p{1pt}cccp{1pt}|}
\cline{2-2}
\cline{6-6}
&&1&1&0&\\
&&2&1&1&\\
&&2&0&1&\\
\cline{2-2}
\cline{6-6}
\end{array} \; ,
\end{equation}
}

\noindent where the lines are labeled (from top to bottom) by $\lambda^1$,
$\lambda^2$ and $\lambda^3$ and the columns (from left to right) by
$\mu^1$, $\mu^2$ and $\mu^3$.
    Note that $D$ is not unitriangularisable (because $D$ has only two $0$
entries). Thus, by Proposition~\ref{prop:unitriangulairerestriction},
$\mathfrak b_{0}$ has no unitriangular $3$-basic set.

\subsubsection{The case of $\alt_{19}$ and $p = 3$}\label{sfock}

    In this section, we show that, in the case of the alternating group
$\alt_{19}$, for $p = 3$, there cannot exist a unitriangular $p$-basic set. 

    To see that, we need the theory of Fock spaces. We briefly recall how
this theory appears in our context and we refer to \cite[Ch. 6]{GJ} for
details. Let $v$ be an indeterminate and define $\mathcal{F}_v$ to be the
$\mathbb{C} (v)$ vector space with basis given by all integer partitions:
$$\mathcal{F}_v :=\bigoplus_{n\in \mathbb{Z}_{\geq
0}}\bigoplus_{\lambda\in \Pi^1(n)} \mathbb{C} (v) \lambda.$$
The module
$\mathcal{F}_v$ is in fact an integrable module over the affine quantum
group of affine type $A$. The action of the divided powers of the Chevalley
operators $E_i$ and $F_i$ can be explicitly given via a combinatorial
formula which we do not recall here. We get the following well-known
proposition whose proof can be found for example in \cite[Prop 2.3]{LR}.
\begin{Prop}\label{fock}
Assume that there exist $k\in \mathbb{N}$,  $(i_1,\ldots,i_k)\in
(\mathbb{Z}/p\mathbb{Z})^k$ and $(a_1,\ldots,a_k)\in \mathbb{N}^k$ such
that  
  $$F_{i_1}^{(a_1)}\ldots F_{i_k}^{(a_k)}.\emptyset = \lambda + \sum_{\mu
    \neq \lambda} a_{\lambda,\mu}(v) \mu$$
with $a_{\mu,\lambda} \in v\mathbb{N}{v}$ for all $\mu \neq \lambda$, then
$\lambda$ is $p$-regular and we have, for all $\lambda\neq \mu$, that
$a_{\mu,\lambda} (1)= d_{\lambda,\mu}$. 
\end{Prop}

    Now, to each $p$-regular partition $\lambda$, by \cite[\S
3.5.11]{GJ}, one can attach a certain sequence of elements in
$\mathbb{Z}/p\mathbb{Z}$:
  $$\eta_p (\lambda):=\underbrace{i_1,\ldots,i_1}_{a_1\text{
     times}},\ldots,\underbrace{i_k,\ldots,i_k}_{a_k\text{ times}},$$
where $a_1+\dots +a_k = n$. This sequence is the analogue of the sequence
obtained using the {\it ladder method} (but this  is not identical.) From
this sequence, we define an element of the Fock space which is
the action of a certain divided power of the Chevalley operators $F_i$ on
the empty partition:
$$A(\lambda):=F_{i_1}^{(a_1)}\ldots F_{i_k}^{(a_k)}.\emptyset \in
\mathcal{F}_v.$$

    Let us now consider the case $n = 19$ and $p = 3$, and the $3$-regular
partition $\lambda=(10,4,4,1)$. First, one can check that $m(10,4,4,1) =
(10,4,4,1)$. Now, we have:
$$\eta_p (\lambda)=(0,2,1,1,0,2,2,1,0,0,0,1,2,2,0,1,1,2,0),$$
which allows the definition of
   $$A (10,4,4,1):=F_0 F_2 F_1^{(2)} F_0 F_2^{(2)} F_1 F_0^{(3)} F_1
     F_2^{(2)} F_0 F_1^{(2)} F_2 F_0 .\emptyset.$$
This element may be computed using the file ``arikidec'' in the directory
``contr''' in the package Chevie \cite{chevie} of gap3 (the function is
called ``aliste''). We get that: 
   
{\small{\small
$$\begin{array}{rcl} A(10,4,4,1)&=& (10,4,1,1)+ v (9, 5, 4, 1 )+  v ( 8, 5, 4, 1, 1 )+  v ( 10, 4, 3, 2 )+ v^2 ( 9, 5, 3, 2 )+ v ( 8, 5, 3, 3 )  +2v^2 ( 8, 5, 3, 2, 1 )\\
  & & +   2v^2 ( 8, 4, 3, 2, 2 )   +  (v^3 + v) ( 7, 5, 4, 3 )+ v^2 ( 7, 4, 4, 4 )+ (v^4 + 2v^2) ( 7, 5, 4, 2, 1 )+ (v^3 + v) ( 7, 4, 4, 2, 2 ) \\
  & & + v ( 7, 6, 5, 1 )+ v^2 ( 6, 6, 6, 1 )+ v^2 ( 7, 6, 4, 1, 1 )+ v^2 ( 7, 6, 3, 3 )+ 
  v^3 ( 7, 6, 3, 2, 1)+v^3 ( 6, 6, 3, 2, 2 )  \\
  & & + v^2 ( 7, 5, 5, 2 )+ v^3 ( 6, 5, 4, 4 )+(v^4 + 2v^2) ( 6, 5, 4, 2, 2 )+
   v ( 10, 3, 3, 2, 1 ) + v ( 9, 3, 3, 2, 2 )\\
   & &  +v^3 ( 8, 3, 3, 3, 2 )+ v^2 ( 10, 3, 3, 1, 1, 1 )+ 2v^2 ( 9, 3, 3, 1, 1, 1, 1 )+
   v^3 ( 8, 5, 3, 1, 1, 1 ) + 2v^3 ( 8, 4, 3, 1, 1, 1, 1 )\\
   & &+(v^4 + v^2) ( 7, 5, 3, 2, 1, 1 )+2v^2 ( 7, 4, 3, 2, 2, 1 )
   3v^3 ( 7, 4, 3, 2, 1, 1, 1 )+ 2v^3 ( 7, 3, 3, 3, 2, 1 )+  \\
   & & v^4 ( 7, 3, 3, 3, 1, 1, 1 ) + v^3 ( 5, 5, 5, 2, 2 )+ v^3 ( 5, 5, 4, 4, 1 )+
  (v^4 + v^2) ( 5, 5, 4, 3, 2 )+ v^3 ( 7, 5, 4, 1, 1, 1 )+\\
  & & 2v^3 ( 6, 5, 4, 1, 1, 1, 1 )+
  v^4 ( 5, 5, 5, 1, 1, 1, 1 ) 2v^3, ( 6, 5, 3, 2, 2, 1 )+ 2v^4, ( 6, 5, 3, 2, 1, 1, 1 )+
   \\
  & & ( 2v^4 + v^2) ( 5, 5, 3, 3, 2, 1 ) + (v^5 + v^3) ( 5, 5, 3, 3, 1, 1, 1 )+ v^3 ( 5, 4, 4, 4, 2 )+
  ( v^4 + v^2) ( 6, 4, 3, 2, 2, 1, 1 )+
\\
   & &    (2v^4 + v^2) ( 5, 4, 3, 3, 2, 1, 1 ) + v^3, ( 5, 5, 3, 2, 2, 2 )+
    v^3 ( 5, 4, 3, 2, 2, 2, 1 )+
   v^3 ( 4, 4, 4, 4, 2, 1 )+

\\
   & &    ( v^5 + v^3) ( 4, 4, 4, 3, 2, 1, 1 )+  v^4 ( 4, 4, 4, 2, 2, 2, 1 )+  v^2 ( 7, 3, 3, 3, 3 )+
   v^3 ( 5, 5, 3, 3, 3 )+
   v^3 ( 6, 3, 3, 3, 2, 1, 1 )+

\\
   &&      v^4 ( 5, 3, 3, 3, 2, 2, 1 )+ v^4 ( 4, 4, 3, 3, 3, 1, 1 )+   v^4 ( 4, 3, 3, 3, 3, 3 )+
   v^5 ( 4, 3, 3, 3, 3, 2, 1 )+
   v ( 10, 3, 2, 1, 1, 1, 1 )+
 \\ 
 &&     v^2 ( 8, 5, 2, 1, 1, 1, 1 )+ (v^4 + v^2) ( 8, 3, 3, 1, 1, 1, 1, 1 )+ v^2 ( 10, 2, 2, 2, 1, 1, 1 )+
   v^3 ( 9, 2, 2, 2, 2, 1, 1 ),

\\
    &&    
       v^3 ( 8, 2, 2, 2, 2, 1, 1, 1 )+
   v^2 ( 7, 5, 2, 2, 1, 1, 1 )+ v^4 ( 7, 4, 2, 2, 2, 1, 1 )+
   2v^3 ( 7, 3, 3, 2, 1, 1, 1, 1 )+ 
  \\
    && 
       v^4 ( 7, 3, 2, 2, 2, 1, 1, 1 )+
   v^2 ( 10, 3, 1, 1, 1, 1, 1, 1 )+
       v^3 ( 8, 5, 1, 1, 1, 1, 1, 1 )+
   v^3 ( 10, 2, 1, 1, 1, 1, 1, 1, 1 )+
     \\
    && 
     v^3 ( 9, 2, 1, 1, 1, 1, 1, 1, 1, 1 )+
   v^4 ( 8, 2, 2, 1, 1, 1, 1, 1, 1, 1 )+
   v^3 ( 7, 5, 1, 1, 1, 1, 1, 1, 1 )+
   v^4 ( 7, 4, 1, 1, 1, 1, 1, 1, 1, 1 )+
     \\
    && 
   2v^4 ( 7, 3, 3, 1, 1, 1, 1, 1, 1 )+
   v^5 ( 7, 3, 2, 1, 1, 1, 1, 1, 1, 1 )+
   v^2 ( 7, 4, 4, 1, 1, 1, 1 )+
   v^3 ( 5, 5, 4, 1, 1, 1, 1, 1 )+ \\
   &&2v^4 ( 5, 5, 3, 2, 1, 1, 1, 1 )+
   2v^4 ( 5, 4, 3, 2, 2, 1, 1, 1 )+
   v^4 ( 4, 4, 4, 4, 1, 1, 1 )+
   v^5 ( 4, 4, 4, 2, 2, 1, 1, 1 )+\\
  &&
     v^3 ( 6, 3, 3, 2, 2, 1, 1, 1 )+
  v^5 ( 5, 3, 3, 3, 2, 1, 1, 1 )+ 
    v^4 ( 6, 3, 3, 1, 1, 1, 1, 1, 1, 1 )+
   v^5 ( 5, 5, 3, 1, 1, 1, 1, 1, 1 )+
      \\
   && 
   v^5 ( 5, 4, 3, 1, 1, 1, 1, 1, 1, 1 )+
   v^4 ( 4, 4, 3, 2, 2, 1, 1, 1, 1 )+
   v^5 ( 4, 4, 3, 2, 1, 1, 1, 1, 1, 1 )+
   v^5 ( 4, 3, 3, 3, 2, 1, 1, 1, 1 )+ \\
   &&
  v^6 ( 4, 3, 3, 3, 1, 1, 1, 1, 1, 1 )
  \end{array}$$}}

    One can see that it satisfies the hypotheses of Proposition
\ref{fock}.  Now,  for $n=19$, we have three partitions in  $\mathcal
C_{(1)} $ which are $\mu^1:= (6,5,3,2,2,1)$, $\mu^2 = (7,4,3,2,1,1,1)$ and
$\mu^3 = (10,1,1,1,1,1,1,1,1)$.  We thus have $d_{\lambda, \mu^1} = 2$,
$d_{\lambda, \mu^2} = 3$ and  $d_{\lambda, \mu^3} = 0$. Thus, by
Proposition \ref{prop:unitriangulairerestriction}, there is no
unitriangular $p$-basic set for $\alt_{19}$ and $p = 3$. 

\section{Unitriangular basic sets for symmetric algebras}
\label{sec:symalgebras}

    This section concerns the existence and the construction of unitriangular
basic sets in the context of symmetric algebras (see the definition in
\cite{GP} - this thus covers the particular case of finite groups).
Given a unitriangular basic set and an involution on the algebra, we show
how one can construct a new unitriangular basic set which enjoys a nice
``stability'' property with respect to the involution.

\subsection{Setting}

    In this section and the following one, we will make two hypotheses.

\begin{Hyp}\label{hy}
Let $A$ be an integral domain, $L$ be a field, and let
$\theta:A\rightarrow L$ be a ring homomorphism. Let $\mathcal H$ be an
associative 
symmetric $A$-algebra.
\begin{enumerate}
\item We assume that $(\BB,\leq,\Psi)$ is a unitriangular basic
set for $( \mathcal{H},\theta)$ with respect to the bijective map $\Psi: \BB
\to \operatorname{Irr} (L \mathcal{H})$.
\item We assume that we have an algebra automorphism $\phi \in
\operatorname{Aut}( \mathcal{H})$ such that $\phi^2=\text{Id}_ \mathcal{H}$. 
\end{enumerate}
  \end{Hyp}
\medskip

\noindent From this, we will show how one can construct another
unitriangular basic set, different from $\BB$ in general, and which enjoys
nice properties with respect to the automorphism $\phi$. To do this, we
first remark the following.
 \begin{itemize}
  \item If $N$ is a simple $K \mathcal{H}$-module (resp. $L
  \mathcal{H}$-module), then we can twist this module by $\phi$ to obtain a
  new simple $KH$-module (resp. $L \mathcal{H}$-module) $N^{\phi}$. For
  $\lambda \in \Lambda$, we denote by $\lambda' $ the element of $\Lambda$ such that
  $(V^{\lambda})^{\phi} \simeq V^{ \lambda '}$.  For $\lambda \in \BB$, we
  denote by $m (\lambda)$ the element of $\BB $ such that $\Psi (\lambda)^{\phi}\simeq
  \Psi (m (\lambda))$.
  \item By the definition of decomposition maps, there is a compatibility
  with the action by automorphisms on modules, and we thus have, for all
  $\lambda \in \Lambda$ and $\mu \in \BB$: $$d_{\lambda, \Psi
  (\mu)}=d_{\lambda', \Psi (m(\mu))}.$$
 \end{itemize}

    We will need the following two elementary results coming from the
above hypotheses (the first one being an analogue of \cite[Prop 4.2]{BOX}
in our context).

  \begin{lemma}\label{box}
    For all $\lambda \in \BB$, we have $m (\lambda) ' \leq
\lambda $ and $\lambda' \leq m (\lambda) $.
  \end{lemma}
  \begin{proof}
  For $\lambda \in \BB$, we have that    $d_{m  (\lambda),
\Psi (m (\mu))}=1$ and, by the above property:
   $$d_{m  (\lambda), \Psi (m (\lambda))}=d_{m   (\lambda)',
\Psi (\lambda)}.$$
By the definition of unitriangular basic set, this means that $ m
(\lambda)' \leq \lambda $. Now, to conclude, note that $ m
(\lambda)$ is in $\BB$ and that $ m^2 $ is the identity. 
  \end{proof}
  
The following result follows from a direct application of Lemma \ref{box}. 
  \begin{lemma}\label{box2}
   Let $\lambda \in \BB$. 
  \begin{enumerate}\label{mullnou}
    \item If $\lambda=m (\lambda)$, then  we have $\lambda\geq \lambda'$.
    \item If $\lambda=\lambda'$, then  we have $m (\lambda)\geq \lambda$.
\end{enumerate}
  \end{lemma}

\subsection{Constructing new unitriangular basic sets}
\label{subsec:newbasic}
  
    We keep the notation of the previous section. We first define a new total
order on the set $\Lambda$ by setting $\lambda \preceq \mu $ if:
\begin{itemize}\label{newor}
\item $\lambda=\mu$, 
\item or  $\lambda=\mu'$ and $\mu \leq \mu'$, 
\item or if 
  $$\operatorname{Max}_{\leq}(\lambda,
    \lambda')<\operatorname{Max}_{\leq}(\mu, \mu').$$
\end{itemize}
We now define a map:
$$\begin{array}{cccc}
\Theta:& \BB& \to& \Lambda\\
& \lambda & \mapsto & \operatorname{Max}_{\preceq } \{ \lambda,  m (\lambda)'\}.
 \end{array}
 $$
In other words, by Lemma \ref{box}, we have:
 $$\Theta (\lambda)=\left\{ \begin{array}{lr} 
 \lambda           & \text{ if } m (\lambda) \leq  \lambda,\\
 m (\lambda)' & \text{ if } m (\lambda) > \lambda.\\
 \end{array}\right.$$
 
\begin{Prop}
The map $\Theta$ is injective. 
\end{Prop}
\begin{proof}
    Assume that $(\lambda,\mu)\in \BB^2$ satisfies $\Theta (\lambda) =
\Theta (\mu)$.  If $\Theta (\lambda) = \lambda$ and $\Theta (\mu) = \mu$,
or if $\Theta (\lambda) = m (\lambda)'$ and $\Theta (\mu) = m
(\mu)'$, we directly obtain that $\lambda = \mu$, so we may assume that we
have $\Theta (\mu) = m (\mu)'$  so that $\mu < m(\mu)$, and
$\Theta (\lambda) = \lambda$ so that $\lambda \geq m (\lambda)$. We
thus have $m (\mu)' = \lambda$. We show that this case is not
possible. 

    Indeed, by Lemma \ref{box}, we have $\mu \geq m(\mu)'$ so $m
(\mu) > m (\mu)'$. On the other hand, as  $\lambda \geq  m
(\lambda)$, we have $m (\mu)'\geq   m (m (\mu)') $. Now, by
Lemma \ref{box}, since $\lambda=m (\mu)'$ is in $\BB$, we have  $
m (m (\mu)') \geq m (\mu)$ and thus $m (\mu)\leq
m (\mu)'$, which is a contradiction.  We must thus have $\lambda=\mu$,
and $\Theta$ is injective.  
\end{proof}
From this, we can prove the main result of this section.

\begin{Th}
\label{main} 
Under the Hypothesis~\ref{hy},   set $\widetilde{\BB}:=\Theta (\BB)$ and 
  $$\widetilde{\Psi}: \widetilde{\BB} \to \operatorname{Irr} (L\mathcal{H})$$
such that, for all $\lambda \in \widetilde{\BB}$, we have $\widetilde{\Psi}(\lambda)
= \Psi (\Theta^{-1} (\lambda))$. Then $(\widetilde{\BB}, \preceq, \widetilde{\Psi})$ is
a unitriangular basic set for $(\mathcal{H},\theta)$. 
\end{Th}

\begin{proof}
  Assume that $\lambda \in \BB$. 
 \begin{itemize}
 \item Assume that $\Theta (\lambda) = \textrm{Max}_{\preceq } (\lambda,
m (\lambda)')=\lambda$. This means that we have $$\textrm{Max}_{\leq}
(\lambda, \lambda')\geq \textrm{Max}_{\leq} ( m (\lambda), m
(\lambda)').$$ As $\lambda$ is in $\BB$, we already know that
$d_{\lambda,\Psi(\lambda)}=1$. In addition, we have
$d_{\mu,\Psi(\lambda)}=0$ unless $\mu \leq {\lambda}$. We also have
$d_{\nu,\Psi(m (\lambda))}=0$ unless $\nu \leq m (\lambda)$.

       Assume that  $d_{\mu,\widetilde{\Psi}(\lambda)} \neq 0$.  We then have
$\widetilde{\Psi} (\lambda)=\Psi (\lambda)$, so $d_{\mu,\Psi(\lambda)} \neq 0$.  We
need to show that  $\mu \preceq \lambda$. We  have $\mu \leq  {\lambda}$,
and thus $\mu \leq \textrm{Max}_{\leq} (\lambda, \lambda')$.  We also have
$d_{\mu',\Psi(m (\lambda))} \neq 0$, and thus $\mu' \leq  m
(\lambda)$. But we have $m (\lambda)\leq \textrm{Max}_{\leq}
(\lambda, \lambda')$. Thus the result follows in this case.

  \item Assume on the other hand that 
 $$\Theta (\lambda) = \textrm{Max}_{\preceq} (\lambda, m (\lambda)')=
   m(\lambda)'.$$
       This means that we have $\textrm{Max}_{\leq} (\lambda,
\lambda')\leq \textrm{Max}_{\leq}  (m (\lambda), m (\lambda)')$.
We already know  $d_{{m (\lambda)},\Psi({m (\lambda)})}=1$
because $m (\lambda)$ is in $\BB$, and thus that $d_{{m
(\lambda)'}{\lambda}}=1$. 
 
    Assume that $d_{\mu,\widetilde{\Psi}(\Theta (\lambda))} \neq 0$. We thus have
$d_{\mu,\Psi(\lambda)} \neq 0$, and then we need to show that $\mu \preceq
\Theta
(\lambda)$, or in other words that $\mu \preceq m (\lambda)$. We
must have $\mu \leq {\lambda}$, which implies that $\mu \leq
\textrm{Max}_{\leq} (\lambda, \lambda')\leq \textrm{Max}_{\leq}
(m(\lambda), m (\lambda)')$. On the other hand, we have
$d_{{\mu',\Psi(m (\lambda))}}\neq 0$, and thus $\mu'\leq m
(\lambda)$, which implies that $\mu' \leq \textrm{Max}_{\leq}
(m(\lambda), m (\lambda)')$. This means that $\mu \preceq m
(\lambda)$, whence the result.
 \end{itemize}
 This proves the claim.
  \end{proof}
Let us now give some consequences of the above theorem.  First, we get a
new classification of the set of simple modules for $L\mathcal{H}$:
  $$ \operatorname{Irr} (L \mathcal{H}):=\{\widetilde{\Psi}(\lambda)  \ |\ \lambda \in 
     \widetilde{\BB}\}.$$
For $\lambda \in   \widetilde{\BB}$, we denote by $\widetilde{m} (\lambda)$ the element of $\widetilde{\BB}$ such that $\widetilde{\Psi} (\lambda)^{\phi}\simeq \widetilde{\Psi} 
(\widetilde{m}(\lambda))$.
   
  \begin{itemize}
   \item  Assume that $M\in \operatorname{Irr} (L \mathcal{H})$ is such that
$M^{\phi} \simeq M$. Set $\lambda:=\Phi^{-1} (M)$. Then $m
(\lambda)=\lambda$, and thus $\Theta (\lambda)=\lambda$. 
   \item  Assume that $M\in \operatorname{Irr} (L \mathcal{H})$ is such that
$M^{\phi}$ is not isomorphic to  $M$. Set $\lambda:=\Phi^{-1} (M)$. Then
$\Theta (m (\lambda))=\Theta (\lambda)'$. 
   \end{itemize}
   As a conclusion, if $\lambda \in \widetilde{\BB}$, then we have 
   $$\widetilde{m} (\lambda)=\left\{ \begin{array}{lc} 
    \lambda & \text{if }  \widetilde{\Psi}(\lambda)^{\phi} \simeq \widetilde{\Psi}(\lambda)\  (\iff m(\lambda)=\lambda),\\
   \lambda' & \text{otherwise}.
  \end{array} \right.$$
  Note that, in the latter case, we have $\lambda'\neq \lambda$ by Lemma \ref{box2}, but that we have also $\lambda'\neq \lambda$  in the first case in general.
  
   Thus, the basic set $\widetilde{\BB}$ allows a parametrization
of the set of simple $L \mathcal{H}$-modules where the twist by the involution is
more convenient to understand on the parametrization. In terms of
$\BB$, we have:
   $$\widetilde{\BB}=\{ \lambda\ |\ \lambda \in \BB,\
     m (\lambda)< \lambda \} \cup \{ \lambda'\ |\ \lambda \in
     \BB,\ m (\lambda)< \lambda \} \cup \{ \lambda\ |\ \lambda
     \in \BB,\ m (\lambda)=\lambda \}.$$
We present a number of applications below. 

\section{Consequences for symmetric groups and their Hecke algebras}
\label{sec:application}

    The aim of this section is to apply the previous section to the case of
the Hecke algebra of type $A$, and to the case of the symmetric and
alternating groups.

\subsection{General case}\label{gene}

    We now apply our result to the case of the Hecke algebra of the
symmetric group. This will be a first application of Theorem \ref{main}. 

    Let $n$ be a positive integer. We write $\mathcal{H}$ for the Hecke
algebra of the symmetric group $\sym_n$ over a commutative ring $R$ with
unit, and let $q\in R$ be invertible. We have a presentation of
$\mathcal{H}$ by:
\begin{itemize}
\item generators: $T_s$,  where $s\in S:=\{s_1,\ldots,s_{n-1}\}$;
\item relations:  $T_{s_i} T_{s_j}=T_{s_j} T_{s_i}$ if $|i-j|>1$, $T_{s_i}
T_{s_{i+1}}T_{s_i}=T_{s_{i+1}} T_{s_i}T_{s_{i+1}} $ for all $i\in
\{1,\ldots,n-2\}$, and the relation  $(T_i-q)(T_i+1)=0$ for all
$i=1,\ldots,n-1$.
\end{itemize}
    We assume that we have a specialization $\theta: R \to L$, where $L$
is a field. We denote 
$$e:= \operatorname{min} (i
\in \mathbb{N}\ | \ 1+\theta(q)+\ldots +\theta (q)^{i-1}=0 ),$$
and assume that $e\in \mathbb{N}_{>1}$. In this case, we are in the
setting of the previous section. 
\begin{itemize}
 \item It is known that the $\mathcal H$-simple modules are naturally
 indexed by the set $\Lambda$ of all partitions of rank $n$, that is, non
 increasing sequences $\lambda=(\lambda_1,\ldots,\lambda_r)$ of integers
 of total sum $n$. \item $\BB$ is the set $\mathcal R_n^e$  of $e$-regular
 partitions, that is, the set of partitions
 $\lambda=(\lambda_1,\ldots,\lambda_r)$ of $n$  such that there exist no
 $i\in \mathbb{N}$ such that $\lambda_i=\ldots=\lambda_{i+e-1}\neq 0$. 
 \item $\leq$ is any refinement of the dominance order on partitions; for
example, take $\leq$ to be the lexicographic order on partitions. 
 \item $\phi$ is the algebra  automorphism sending $T_i$ to
$-T_i+(q-1)$ for all $i=1,\ldots,n-1$.
 \item For $\lambda \in \Lambda$, it is well-known that  $\lambda'$ is the
 usual conjugate partition of $\lambda$, that is, the partition obtained
 by reflecting the Young diagram of the partition $\lambda$ along its main
 diagonal.
 \item Finally, $m$ is known as the Mullineux involution, which we still denote
by $m$. It can be computed using different recursive
algorithms which we don't recall here \cite{BOX,K,Mu}.  
\end{itemize}

    Applying the main result of the previous section thus yields the
following result: 
 
 \begin{Prop}\label{bset}
 $( \mathcal{H},\theta)$ admits a unitriangular basic set $(\widetilde{\BB},\preceq, \widetilde{\Psi})$, where $\preceq$ 
  is defined from $\leq$ in (\ref{newor}), $\widetilde{\Psi}=\Psi \circ \Theta^{-1}$, and 
  $$\widetilde{\BB}=\{ \lambda\in \mathcal R_n^p \ |\ 
    m(\lambda)< \lambda \} \sqcup \{ \lambda'\ |\ \lambda\in \mathcal R_n^p,\ m
    (\lambda)< \lambda \} \sqcup \{ \lambda\in \mathcal R_n^p\ |\
    m(\lambda)=\lambda \}.$$
 \end{Prop}
 
    We thus obtain a new classification of the set of simple modules.
Comparing to the usual classification given by the $e$-regular partitions,
the main advantage of this classification is that the twist by the
automorphism is easy to read on it. 
 \begin{itemize}
  \item  Assume that $e=2$. In this case, we have, for all $2$-regular
partitions, $m(\lambda)=\lambda$, because the involution $m$ is the
identity. This happens for example when $q=1$ and $L$ is a field of
characteristic $2$. In this case, the Hecke algebra is nothing but the group algebra of the
symmetric group over this field.  We see in this case that
$\BB=\widetilde{\BB}$. However, the orders $\leq$ and
$\preceq$ are of course different.
   \item The most interesting case is the case where $e\neq 2$. Indeed,
the sign representation 
is labeled by the partition $(1,\ldots,1)$, and this label lies  
in $\widetilde{\BB}$ as in the semisimple case. However, if
we want to obtain an explicit description of $\widetilde{\BB}$,
we need to compute the Mullineux involution for each $e$-regular
partition, which is a long, recursive and non-trivial problem. It would
thus be desirable to describe the $e$-regular  partitions satisfying
$m(\lambda)\leq \lambda$ and $m(\lambda)= \lambda$ without computing
$m(\lambda)$. We hope to come back to this problem in future work. 
\end{itemize}
    The above result can in particular be applied to the case where $q=1$
and $R$ is a field of characteristic $p>2$, so that $e=p$. The algebra $H$
is then  nothing but the group algebra of the symmetric group over a field
of characteristic $p$. It is now natural to ask if the unitriangular basic
set satisfies the hypothesis of Theorem \ref{thm:basicGH}. It should not
be the case, otherwise we would always have a unitriangular basic set for
the alternating groups, and we have already seen that this is not the
case. The problem comes from the fact that our basic set is not stable
with respect to the sign representation. In fact, from Lemma
\ref{mullnou}, we deduce:
\begin{itemize}
\item  if $\lambda \neq m(\lambda)$, then both $\lambda$ and $\lambda'$
are in the basic set $\widetilde{B}$, and we do  have $\widetilde{\Psi}
(\lambda')=\varepsilon \otimes \widetilde{\Psi} (\lambda)$;
\item  if $\lambda = m(\lambda)$, then $\lambda'$ is not  in
$\widetilde{\BB}$.
\end{itemize}

\begin{exa}
    Take $p=3$ and $n=5$. Then the set of $3$-regular partitions of $n$ is:
$$\BB=\{ (2, 2, 1), (3, 1, 1), (3, 2 ), (4, 1), (5)\}.$$ 
Here we consider the lexicographic order on partitions. For each of these
partitions, we compute the image under the Mullineux involution to find
our new basic set $\widetilde{\BB}$ (computed with respect to the
lexicographic order).  We have $m(2,2,1) = (4,1)>(2,2,1)$,
so that $(4,1) \in \widetilde{\BB}$ and $(2,1,1,1)=(4,1)' \in \widetilde{\BB}$. 
    We have $m(3,1,1) = (3,1,1) \in \widetilde{\BB}$. And
$m(3,2)=(5)>(3,2)$, so that $(5) \in \widetilde{\BB}$ and $(1,1,1,1,1)=(5)' \in
\widetilde{\BB}$. We get:
  $$\widetilde{\BB}=\{ (5), (1,1,1, 1, 1), (3, 1,1), (4, 1), (2,1,1,1)\}.$$
    We can see that, in this particular case, the unitriangular basic set is
stable with respect to the sign and $\varepsilon$-equivariant. 
\end{exa}

 
\begin{exa}
    Take $p=3$ and $n=8$. Again using the lexicographic order, we get:
  $$\widetilde{\BB}=\{ (4, 2, 1, 1 ), ( 2, 2, 2, 1, 1 ), ( 2, 2, 1, 1, 1, 1 ), 
  (3, 1, 1, 1, 1, 1 ), ( 3, 2, 1, 1, 1 ), ( 2, 1, 1, 1, 1, 1, 1 ), 
  ( 1, 1, 1, 1, 1, 1, 1, 1 ),$$
  $$ (5, 2, 1 ), ( 5, 3 ), ( 6, 1, 1 ),( 6, 2 ), 
  ( 7, 1 ), ( 8 ) \}.$$
    Note that the conjugate of $(4, 2, 1, 1 )$ is not in the basic set.
Thus, the unitriangular basic set isn't stable with respect to 
tensoring by the sign character.
\end{exa}
 
 \begin{Rem}\label{Rorder}
    In this section, we have considered the lexicographic order for our
examples, but we can in fact choose any total order which is compatible
with the dominance order. For example, we can take the following one. For
$(\lambda,\mu) \in \Lambda^2$
 $$\lambda \leq' \mu \iff \lambda' \geq \mu',$$
where $\geq$ is the lexicographic order. Quite remarkably, this order
gives a different basic set in general. For example, the partition
$\lambda=(6,2,2,1,1 )$ is in the basic set associated with the
lexicographic order when $p=3$, because this is a $3$-regular partition and
we have $m(\lambda)=(5,3,2,2)$. So we have $m(\lambda )\leq \lambda$.
However we see that $m(\lambda )\geq' \lambda$ so that $(5,3,2,2)$ (and
its conjugate) are in the unitriangular basic set associated with $\leq'$,
while $(6,2,2,1,1)$ is not. 
 \end{Rem}

    Of course, all of the above results make sense if we restrict ourselves to
blocks of the symmetric group. Let $\gamma$ be a $p$-core and assume that
we have constructed our unitriangular basic set
$(\widetilde{\BB}_{\gamma}, \preceq, \widetilde{\Psi})$ with respect to an
order associated to the block $\mathfrak{B}_{\gamma}$ of the symmetric
group. Then we have :
  $$\widetilde{\BB}_{\gamma}=\widetilde{\BB}_{\gamma}^1
  \sqcup\widetilde{\BB}_{\gamma}^2,$$
  where 
  $$\widetilde{\BB}_{\gamma}^1 =\{ \lambda\in \mathcal R_n^p \ |\ 
    m(\lambda)< \lambda \} \sqcup \{ \lambda'\ |\ \lambda\in \mathcal R_n^p,\ m
    (\lambda)< \lambda \} $$
   $$ \mbox{and} \quad \widetilde{\BB}_{\gamma}^2 =\{ \lambda\in \mathcal R_n^p\ |\
    m(\lambda)=\lambda \}.$$
    In the final two subsections, we will discuss two favorable cases
    which allow the construction of a unitriangular basic set respecting the
    hypotheses  of Theorem \ref{thm:basicGH}:
\begin{itemize}
\item Subsection \ref{subsec:consAn} concerns a case where
$\widetilde{\BB}_{\gamma}^2=\emptyset$. 
\item In Subsection \ref{subsec:last}, we study a case where there exists a bijection
$$\rho:\widetilde{\BB}_{\gamma}^2 \to \mathcal{G}_{\gamma},$$ where
$\mathcal{G}_{\gamma}$ is the set of self-conjugate partitions whose
diagonal $p$-hooks have length not divisible by $p$, satisfying the following two
properties:
\begin{itemize}
\item For all $\mu \in \widetilde{\BB}_{\gamma}^2$, we have
$d_{\rho(\mu),\mu}=1$.
\item For all $\mu \in \widetilde{\BB}_{\gamma}^2$  and for all $\lambda
\in \widetilde{\BB}_{\gamma}$ such that $\rho (\mu)\prec \lambda \prec
\mu$, we have $d_{\rho (\mu),\lambda}=0$.
\end{itemize}
    Then, in this case, by  \cite{Be}, it follows from the form of the
decomposition matrix that $(\widetilde{\BB}_{\gamma}',\preceq,
\widetilde{\Psi}')$ is a unitriangular $p$-basic set for
$\mathfrak{B}_{\gamma}$, where 
$$\widetilde{\BB}_{\gamma}'=\widetilde{\BB}_{\gamma}^1 \sqcup \rho(
\widetilde{\BB}_{\gamma}^2),$$
and, for all $\lambda \in \widetilde{\BB}_{\gamma}'$, we have:
 $$\widetilde{\Psi}'(\lambda)=
 \left\{\begin{array}{ll}
 \widetilde{\Psi} (\lambda ) & \text{if } \lambda\in \widetilde{\BB}_{\gamma}^1,\\
  \widetilde{\Psi} (\rho^{-1}(\lambda )) & \text{if } \lambda\in \rho(\widetilde{\BB}_{\gamma}^2).\\
 \end{array}\right.$$
From this, one can deduce a unitriangular $p$-basic set for
$\mathfrak{b}_{\gamma}$ by Theorem \ref{thm:basicGH}.
\end{itemize}

Of course, the above unitriangular basic set $\widetilde{\BB}$ satisfies
the assumptions of Theorem \ref{thm:basicGH} if no self-Mullineux $e$-regular partition exists. This is exactly what happens for the
blocks considered in the next subsection. There is another way to obtain
the desired unitriangular basic set which is explained in \cite{Be}.

\subsection{Blocks with odd weight}
\label{subsec:consAn}

\begin{Th}
\label{thm:poidsimpair} Let $p$ be an odd prime number. If $\mathfrak b_{\gamma}$ is
a $p$-block of $\alt_n$ with an odd $p$-weight, then $\mathfrak b_{\gamma}$ has a
unitriangular $p$-basic set. 
\end{Th}

\begin{proof}
    Let $\mathfrak B_{\gamma}$ be the $p$-block of $\sym_n$ covering
$\mathfrak b_{\gamma}$. We only have to consider the case where $\mathfrak
B_{\gamma}$ is $\varepsilon$-stable. First, we remark that 
$$\mathfrak B_{\gamma}\cap \mathcal T = \emptyset.$$
    Indeed, any character $\chi_{\lambda}\in\mathfrak B_{\gamma}\cap
\mathcal T$ satisfies $\lambda_{(p)} = \gamma$ and
$\lambda^{(p)}\in\mathcal C_{\gamma}$. In particular, if $w$ is the
$p$-weight of $\mathfrak B_{\gamma}$, then
$$w = \sum_{i=1}^{p}|\lambda^{i}| = 2\sum_{i=1}^{(p-1)/2}|\lambda^{i}|,$$
which contradicts the hypothesis that $w$ is odd.
 If then $(\widetilde{\BB}_{\gamma},
\preceq,\widetilde{\Psi})$ is a unitriangular basic set constructed from the
unitriangular $p$-basic set indexed by the $p$-regular partitions (using for example the lexicographic order),
  then
$\widetilde{\BB}_{\gamma} = \widetilde{\BB}_{\gamma}^1$, whence, by the above discussion, this 
unitriangular $p$-basic set  restricts to a unitriangular basic set of $\mathfrak b_{\gamma}$, as required.
\end{proof}

\begin{Rem}
The condition of Theorem~\ref{thm:poidsimpair} is not necessary. Indeed, we
see in Example~\ref{ex:S6A6} that the principal $3$-block of $\alt_6$, which has $3$-weight $2$, has a unitriangular $3$-basic set.
\end{Rem}

\subsection{The case of $\sym_{23}$ and $p=3$}
\label{subsec:last}
    We now consider the case $n=23$ and $p=3$, and the $p$-block associated to
the $3$-core $(3,1,1)$. We also consider the order on partitions given in
Remark \ref{Rorder} (this is thus not the lexicographic order). 

    The unitriangular basic set of the block of the symmetric group with
core $(3,1,1)$ in Proposition \ref{bset} contains  $65$ characters:
\begin{itemize}
\item $31$ characters labeled by the $3$-regular partitions
$\lambda$ such that $\lambda > m(\lambda)$,
\item $31$ characters labeled by  the conjugates of the above partitions,
\item $3$ characters labeled by the partitions which are stable with
respect to the Mullineux involution; these are $(12,6,5)$, $(10, 4, 4, 3,
1, 1 )$ and $( 9, 6, 3, 3, 1, 1 )$.
\end{itemize}
    We denote by $\widetilde{\BB}^1_{(3,1,1)}$ the partitions labeling the
characters of the first two sets of characters above, and by
$\widetilde{\BB}^2_{(3,1,1)}$ the partitions labeling the  three
characters of the last set, so that 
$$\widetilde{\BB}^2_{(3,1,1)}=\{(12,6,5),(10, 4, 4, 3, 1, 1 ),( 9, 6, 3,
3, 1, 1 )\}.$$
We now focus on the elements of $\operatorname{IPr}_p(G)$ labeled by these last three partitions.  
As in section \ref{sfock}, we consider the following element of the Fock
space $\mathcal{F}_v$:
 $$A(9,6,3,3,1,1)=F_2 F_1^{(2)} F_0^2 F_2^{(3)} F_1 F_0^2 F_2 F_1^{(2)}
 F_2 F_0^{(2)} F_1 F_2^{(2)} F_0 F_1 F_2 F_0 .\emptyset,$$
which gives the following element of the Fock space:
{\small
$$\begin{array}{rcl}  A(9,6,3,3,1,1)&=&
(9,6,3,3,1)+v (8,7,3,3,1,1)+v (9,5,4,3,1,1)+v^2 (8,5,5,3,1,1) +v^2 (7,7,4,3,1,1))\\
&+&v (9, 4, 4, 3, 2, 1 )+ (v^3 + v)  (8, 4, 4, 3, 2, 2 ) + v^2 (9, 4, 3, 3, 3, 1)+v^2 (8, 4, 3, 3, 3, 2 )+ v^3 (7, 4, 4, 3, 3, 2)\\
 &+&v^4 + v^2)  (8, 4, 4, 3, 1, 1, 1, 1 )+ v^2 (6, 6, 5, 3, 2, 1)+(v^4 + v^2) (6, 6, 4, 3, 2, 2 ) +v^3 (6, 6, 3, 3, 3, 2 )\\
  &+&v^4 (6, 5, 4, 3, 3, 2 )+v^3 (6, 6, 5, 3, 1, 1, 1)+( v^5 + v^3)  (6, 6, 4, 3, 1, 1, 1, 1 )+ v^3 (6, 4, 4, 3, 3, 2, 1)\\
 &+&v (9, 6, 2, 2, 2, 2 ) +v^2 (8, 7, 2, 2, 2, 2 )+v^2 (9, 4, 4, 2, 2, 2 )+v^3 (9, 4, 2, 2, 2, 2, 2 )+v^3 (8, 4, 2, 2, 2, 2, 2, 1)\\
  &+& v^3 (6, 6, 5, 2, 2, 2 )+v^4 (6, 6, 2, 2, 2, 2, 2, 1)+ v^4 (6, 4, 4, 3, 2, 2, 2 )+ v^5 (6, 4, 4, 2, 2, 2, 2, 1 ) \\
  &+&v^3 ( 8, 7, 2, 2, 1, 1, 1, 1 ) +v^2 ( 9, 5, 2, 2, 1, 1, 1, 1, 1 )+ v^3 (7, 7, 2, 2, 1, 1, 1, 1, 1)+v^3 (9, 4, 4, 2, 1, 1, 1, 1 )+\\ 
  &+& v^3 (8, 4, 3, 3, 1, 1, 1, 1, 1 )+v^4 (7, 4, 4, 3, 1, 1, 1, 1, 1 ) +v^4  (9, 4, 2, 2, 2, 1, 1, 1, 1 )+ v^4 ( 8, 4, 2, 2, 2, 2, 1, 1, 1 )\\
&+&v^4 (6, 6, 3, 3, 1, 1, 1, 1, 1 )+v^4 (6, 5, 5, 2, 1, 1, 1, 1, 1)+ v^5 (6, 5, 4, 3, 1, 1, 1, 1, 1 )+v^5 (6, 6, 2, 2, 2, 2, 1, 1, 1 ) \\
  &+&  v^6 (6, 4, 4, 2, 2, 2, 1, 1, 1 )+v^2 (9, 4, 4, 3, 1, 1, 1)+ v^3 (6, 5, 5, 3, 3, 1 )+v^4 (6, 4, 4, 3, 3, 1, 1, 1)\\
  &+&v^3 (7,6,5,3,1,1)+v^2 (9, 6, 2, 2, 1, 1, 1, 1 )+v^3( 9, 4, 3, 2, 1, 1, 1, 1, 1 )+v^3( 9, 4, 3, 2, 1, 1, 1, 1, 1 )\\
  &+&v^2 (9, 6, 2, 2, 1, 1, 1, 1 )+v^4 (6, 6, 5, 2, 1, 1, 1, 1 )+v^5( 6,
  4, 4, 3, 2, 1, 1, 1, 1 )+v^3 (7,6,5,3,1,1).
\end{array}$$}

    The assumptions of Proposition \ref{fock} are satisfied. This means in
particular that the decomposition number $d_{(6,5,5,3,3,1),(9,6,3,3,1,1)}$
is $1$. Note that  $(6,5,5,3,3,1)$ is in $\mathcal{G}_{(3,1,1)}$. 

    Now let us consider
 $$A(10,4,4,3,1,1)=F_1 F_2 f_1^3 F_0 F_3^{(4)} F_1 F_0^{(2)}  F_1^{(2)} F_2^{(2)} F_0^{(2)}  F_1^{(2)} F_2 F_0 .\emptyset,$$
which gives the following element of the  Fock space:
{\small
$$\begin{array}{lll}  A(10,4,4,3,1,1)&=&(10,4,4,3,1,1)+v(9,5,4,3,1,1)+v^2(9,4,4,4,1,1)+v^3 (9,4,4,3,2,1)+v^4 (9,4,4,3,1,1,1) \\
&+&v(7,6,5,3,1,1)+v^2 (6,6,6,3,1,1)+v^3 (6,6,5,4,1,1)+v^4(6,6,5,3,2,1)+v^5 (6,6,5,3,1,1,1)\\
&+&v(7,4,4,3,3,2)+v^2 (6,5,4,3,3,2)+v^3(6,4,4,4,3,2)+v^4 (6,4,4,3,3,3)+v^5 (6,4,4,3,3,2,1)\\
&+&v(10,4,2,2,1,1,1,1,1)+v^2(9,5,2,2,1,1,1,1,1)+v^3(9,4,3,2,1,1,1,1,1)+v^4(9,4,2,2,2,1,1,1,1)\\
&+&v^5 (9,4,2,2,1,1,1,1,1,1)+v^2(7,4,4,3,1,1,1,1,1)+v^3 (6,5,4,3,1,1,1,1,1)+v^5(6,4,4,3,2,1,1,1,1)\\
&+&v^5 (6,4,4,3,2,1,1,1,1)+v^6 (6,4,4,3,1,1,1,1,1)+v^5(6,4,4,3,2,1,1,1,1).
\end{array}$$}
    Again, the assumptions of Proposition \ref{fock} are satisfied. This
means in particular that the decomposition number $d_{(9,4,3,2,1,1,1,1,1),
(10,4,4,3,1,1)} $ is $1$ and $(9,4,3,2,1,1,1,1,1)$ is in
$\mathcal{G}_{(3,1,1)}$. 

    Now, let us consider the partition $(12,6,5)$ and the partition
$(12,1,1,1,1,1,1,1,1,1,1,1)$ which is in  $\mathcal{G}_{(3,1,1)}$. Note
that the regularization of this last partition is exactly given by
$(12,6,5)$, so that the associated decomposition number
$d_{(12,1,1,1,1,1,1,1,1,1,1,1),  (12,6,5)  }$ is also $1$.
 

    Now, we will use the same argument as the one used in
\cite{Be}, and already explained in \S \ref{gene}. Let $\mu$ be one of the
above three $3$-regular partitions  fixed by the Mullineux involution.
For each of them, we have found an element $\rho (\mu):=\nu \in
C_{(3,1,1)}$ such that $d_{\nu,\widetilde{\Psi} (\mu) }=1$.  Assume that,
for all  $\lambda$ in  $\widetilde{\BB}$ such that $\nu \prec \lambda
\prec \mu$, we have $d_{\nu,\widetilde{\Psi}(\lambda)}=0$. Here, we have
$\rho (12,6,5) = (12,1,1,1,1,1,1,1,1,1,1,1)$, $\rho (9,6,3,3,1,1) =
(6,5,5,3,3,1)$ and $\rho (10,4,4,3,1,1) = (9,4,3,2,1,1,1,1,1)$. Then, by
\cite{Be}, we have that 
  $$(\widetilde{\BB}_1 \sqcup\{ (9,4,3,2,1,1,1,1,1),(12,1,1,1,1,1,1,1,1,1,1,1),(6,5,5,3,3,1)\}, \preceq, \widetilde{\Psi}')$$
gives a unitriangular basic set for the block which satisfies the
assumption of Theorem \ref{thm:basicGH}. Indeed, in the three cases, the
decomposition number $d_{\nu,\widetilde{\Psi}(\lambda)}$ is $d_{\nu,\Psi (
\lambda)}$ if $\lambda$ is $p$-regular and $\lambda \geq m(\lambda)$, and
$d_{\nu,\Psi (m(\lambda'))}=d_{\nu,\Psi (\lambda')}$ otherwise. It is zero
if $\lambda$ does not dominate $\nu$ or $\lambda'$ does not dominate
$\nu$. This property is always satisfied, which can be seen using a direct
computation, except in the case where $\nu=(6,5,5,3,3,1)$, $\mu=
(9,6,3,3,1,1)$ and $\lambda= ( 10, 4, 4, 3, 1, 1) $. But we know the
decomposition number  $d_{\nu,\Psi ( \lambda)}$ (see above) in this case,
and we see that it is zero.

    It thus gives a unitriangular basic set for the associated block of
the group $\alt_{23}$. Furthermore, by \cite{CR}, the blocks of the
symmetric group with the same weights are all derived equivalent. The
above result  together with the example studied in \S \ref{exA18} show
that the set of conditions of Theorem \ref{thm:basicGH}  is not an
invariant under this equivalence.

\bigskip
{\bf Acknowledgements.} The first and third authors acknowledge the
support of the ANR grant GeRepMod ANR-16-CE40-0010-01. The third author is also supported by  ANR project JCJC ANR-18-CE40-0001. 
 The authors
sincerely thank Gunter Malle for his precise reading of the paper and his
helpful comments.


\noindent {\bf Addresses}\\

\noindent \textsc{Olivier Brunat},
Institut Math\'ematique de Jussieu - Paris Rive Gauche,
\'{E}quipe : Groupes, repr\'esentations et g\'eom\'etrie,
Case 7012, 75205 Paris Cedex 13,
FRANCE.\\
\emph{olivier.brunat@imj-prg.fr}

\noindent \textsc{Jean-Baptiste Gramain}, Institute of Mathematics, University of Aberdeen, King's College, Fraser Noble
Building, Aberdeen AB24 3UE, UK\\
\emph{jbgramain@abdn.ac.uk}

\noindent \textsc{Nicolas Jacon}, Universit\'e de Reims Champagne-Ardenne, UFR Sciences exactes et naturelles, Laboratoire de Math\'ematiques FRE 2011
Moulin de la Housse BP 1039, 51100 Reims, FRANCE\\  \emph{nicolas.jacon@univ-reims.fr}

\end{document}